\documentclass[12pt,a4paper]{article}
\usepackage[utf8]{inputenc}
\usepackage[english]{babel}
\usepackage{amsmath}
\usepackage{amsfonts}
\usepackage{amssymb}
\usepackage{graphicx}
\usepackage{amsthm,mathrsfs}
\usepackage{graphicx,xcolor}
\usepackage{enumerate}
\usepackage[round]{natbib}
\usepackage[left=4cm,right=4cm,top=4cm,bottom=2cm]{geometry}
\usepackage[hidelinks]{hyperref}

\newtheorem{theorem}{Theorem}[section] 
\newtheorem{lemma}[theorem]{Lemma}     

\newtheorem{proposition}[theorem]{Proposition}

\newtheorem{algorithm}{Algorithm}

\title{Describing elements of the genus-2 Goeritz group of $S^3$}
\author{Sreekrishna Palaparthi, Swapnendu Panda\\ \tiny{Department of Mathematics, Indian Institute of Technology Guwahati}}
\begin{document}
\maketitle

\begin{abstract}
In this  article we present a finite generating set $G_2$ of $\mathcal{H}_2$, the genus-2 Goeritz group of $S^3$, in terms of Dehn twists about certain simple closed curves on the standard Heegaard surface. We present an algorithm that describes an element $\psi\in\mathcal{H}_2$ as a word in the alphabet of $G_2$ in a certain format. Using a complexity measure defined on reducing spheres, we show that such a description of $\psi$ is unique.
\end{abstract}

\section{Introduction} 
\label{intro}

	 The genus $g$ Heegaard splitting of the three sphere is a decomposition of $S^3$ as $V_g \cup_{\Sigma_g} W_g$
	where $V_g$  and $W_g$ are genus $g$ handlebodies in $S^3$ glued along their common boundary $\Sigma_g = \partial V_g = \partial W_g$. If $\Sigma_g$ is the standard unknotted genus $g$ surface, then we call this the standard genus $g$ Heegaard splitting of $S^3$. The set of isotopy classes of orientation preserving homeomorphisms of $S^3$ that leave the standard $\Sigma_g$ invariant naturally forms a group, $\mathcal{H}_g$, and is called the genus $g$  Goeritz group. Since elements of $\mathcal{H}_g$, when restricted to $\Sigma_g$, are elements in the mapping class group of $\Sigma_g$, $\mathcal{MCG}(\Sigma_g)$, $\mathcal{H}_g$ can be thought of as a subgroup of  $\mathcal{MCG}(\Sigma_g)$. This group can also be thought of as the set of elements of $\mathcal{MCG}(\Sigma_g)$, which can be extended to isotopy classes of automorphisms of $S^3$.
	
	The study of Goeritz group of the three sphere dates back to 1930s. Early work in this direction includes \cite{goeritz1933abbildungen} which proved that the $\mathcal{H}_2$ is finitely generated. He also gave a set of four generators. \cite{powell1980homeomorphisms} attempted a generalization of Goeritz's result for higher genus cases. He introduced a set of generators for the Goeritz group $\mathcal{H}_g$. These automorphisms are termed as `Powell generators'. But later on \cite{scharlemann2003automorphisms} identified a gap in Powell's proof. He produced an updated proof for the finite generation of $\mathcal{H}_2$ in 2003 and he established that $\mathcal{H}_2$ is generated by the four automorphisms $\alpha,\beta,\gamma$ and $\delta$ described in \cite{scharlemann2003automorphisms}.

\cite{akbas} extended Scharlemann's work by providing a finite presentation of $\mathcal{H}_2$. He established the acyclic nature of a certain graph $\tilde{\Gamma}$ constructed in \cite{scharlemann2003automorphisms} and using this he gave a finite presentation for $\mathcal{H}_2$.

\cite{cho} produced an alternate proof of the fact that the graph $\tilde{\Gamma}$ in \cite{scharlemann2003automorphisms} and \cite{akbas} is a tree. He used primitive disks and constructed a primitive disk complex $P(V)$. He finally constructed a graph $T$ in the barycentric subdivision of $P(V)$ and showed that $T$ is a tree. He also demonstrated that $T$ and the tree in \cite{akbas} can be reconciled.

\cite{freedman2018powell} proved the finite generation of the Goeritz group $\mathcal{H}_3$ of the genus three Heegaard splitting of the three sphere. They used the generators proposed in  \cite{powell1980homeomorphisms}. They had further conjectured that the same set of generators will generate the Goeritz groups for the higher genus cases. This is called the Powell's conjecture and is still open for genus greater than three.

\cite{zupan2019powell} constructed a curve complex by the reducing spheres on a standard genus $g$ Heegaard splitting surface and studied some relations between the reducing sphere complex and the Powell Conjecture. He showed that Powell conjecture is true if and only if the said reducing sphere complex is connected. Recently \cite{scharlemann2019powell} has announced that one of the Powell generators in \cite{freedman2018powell} is redundant.

Despite being finitely presented, we know how difficult it can be to algorithmically describe every element of a group. Likewise, the algorithms in \cite{scharlemann2003automorphisms}, \cite{akbas} and \cite{cho} do not tell us how to uniquely represent every element of $\mathcal{H}_2$.

In this article, we represent every element of $\mathcal{H}_2$ in a unique way such that no two representations are the same. In showing so, we give yet another proof of finite generation of $\mathcal{H}_2$ using the description of the stabilizer of the standard reducing sphere in \cite{scharlemann2003automorphisms}. We begin by expressing three representatives, $\beta$, $\varphi$ and $\varphi\nu$ of distinct automorphism classes in $\mathcal{H}_2$ as Dehn twists about non-separating curves on the Heegaard surface $\Sigma_2$. To every reducing sphere $Q$, we associate a certain triple of non-negative integers, $T_Q$, of the geometric intersection numbers of the curve $Q \cap \Sigma_2$ with certain curves on $\Sigma_2$. We define a positive integer $\mathcal{C}(Q)$ based on $T_Q$ such that the unique reducing sphere with $\mathcal{C}(Q) =1$ is the standard reducing sphere $P$. Our main result then is an algorithm to write an automorphism in $\mathcal{H}_2$ as a word in the alphabet of $G_2 =\{\beta, \varphi, \nu, \alpha \}$ as follows. Since an automorphism $f$ in $\mathcal{H}_2$ maps the standard reducing sphere $P$ to some reducing sphere $Q$, we start with the reducing sphere $Q$. Using $T_Q$ we give a criteria to determine an automorphism among the four, $\beta$, $\beta^{-1}$, $\varphi$ and $\varphi\nu$, which when applied to $Q$ gives a new reducing sphere $R$ such that $\mathcal{C}(R) <\mathcal{C}(Q)$. We can explicitly calculate the reducing sphere $R$ by applying the Dehn twist expression of the automorphism applied to $Q$. Now we repeat this process for $R$. At each stage we append the automorphism just applied to the word constructed so far. The algorithm terminates when the integer $\mathcal{C}(R)$ reduces to $1$ and $R$ is the standard sphere. So $\mathcal{C}(Q)$ serves as a complexity measure. We show that the automorphism $f$ has the form $$f=\alpha^{a}\nu^{b}\beta^{c}\prod \left(\varphi\nu^{s_i}\beta^{r_i} \right)=\alpha^{a}\nu^{b}\beta^{c}\left(\varphi\nu^{s_n}\beta^{r_n} \right)\circ\cdots\circ\left(\varphi\nu^{s_1}\beta^{r_1} \right) \qquad (*) $$ where $a, b, s_i=0,1$ and $c, r_i\in\mathbb{Z}$. Since the complexity measure is monotonous while applying the automorphisms in $G_2$ in the order as in $(*)$, we conclude that every element in $\mathcal{H}_2$ can be uniquely written in the form $(*)$.

This is part of the thesis work of the second author. He is examining how the techniques in this article can be used to prove finite generation of $\mathcal{H}_g$ for $g\geq3$.

\section{Setup and Preliminaries}\label{setup}

   We refer the reader to \cite{farb} for basic terminology related to mapping class groups of surfaces and \cite{scharlemann2003automorphisms} and \cite{akbas} for terms related to Heegaard splittings. Consider a standardly-embedded genus two surface $\Sigma_2$ in $S^3$. Let $S^3=V_2\cup_{\Sigma_2} W_2$ be the corresponding Heegaard splitting of $S^3$. 

   Consider the curves shown in Figure \ref{gen2-loops} on $\Sigma_2$:
		\begin{figure}[h]
			\centering\includegraphics[width=8cm]{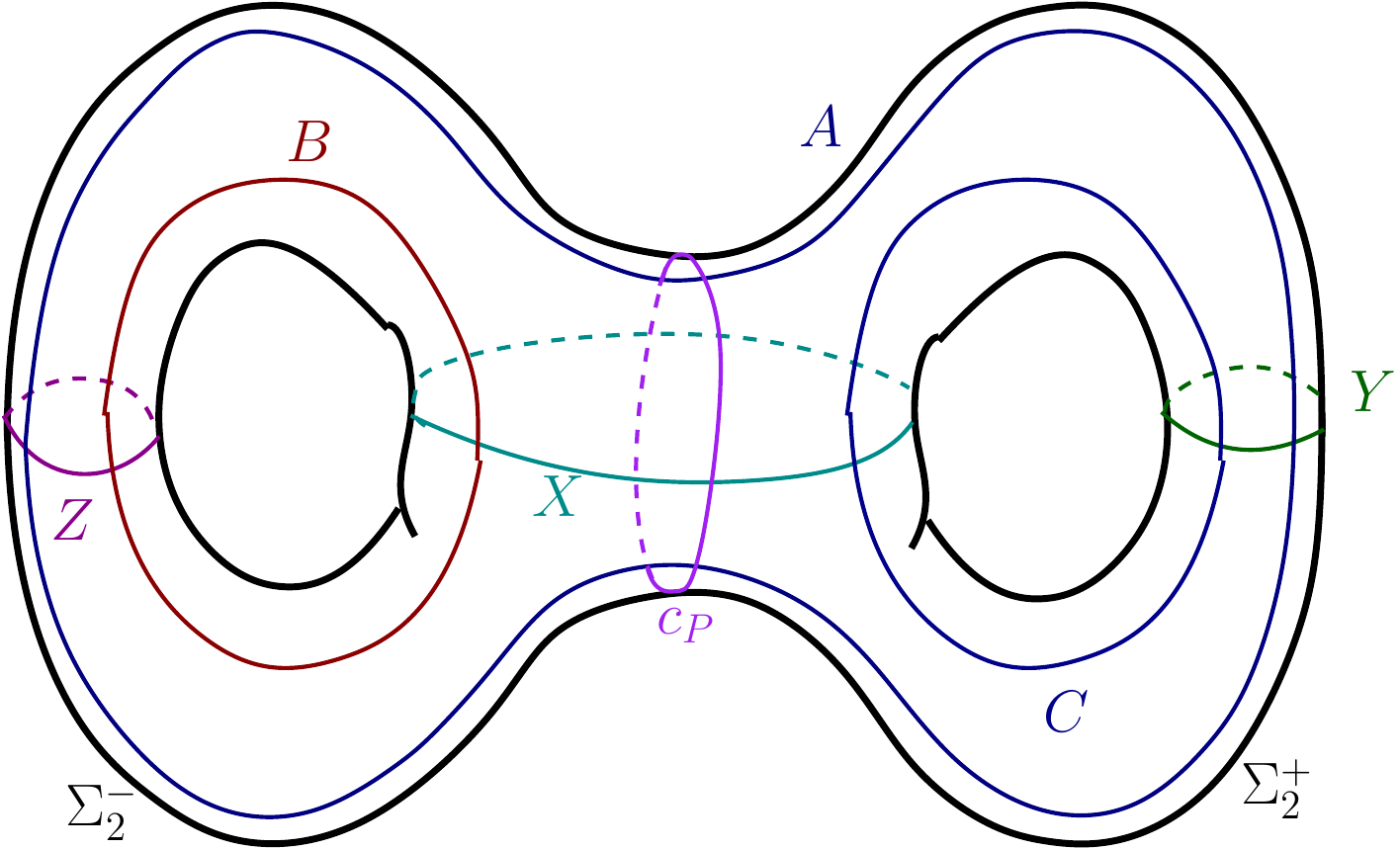}
			\caption{The standard set of curves on $\Sigma_2$}\label{gen2-loops}
		\end{figure}
	$A,B,C,X,Y,Z$ are non-separating curves on $\Sigma_2$. $A\cup B\cup C$ separates $\Sigma_2$ into two thrice boundered spheres, call them $\Sigma_2'$ and $\Sigma_2''$. If $J$ and $K$ are isotopy classes of curves on $\Sigma_2$, then by $J\cdot K$ we mean the geometric intersection of $J$ and $K$. For any reducing sphere $Q$, we call the essential separating circle $c_Q=Q\cap\Sigma_2$ on $\Sigma_2$ as the reducing curve corresponding to $Q$. $P$ is the reducing sphere whose reducing curve is $c_P$ as shown in figure \ref{gen2-loops}. We call $P$ as the standard reducing sphere. $P$ separates $\Sigma_2$ into two genus one surfaces with one boundary. We call these component surfaces as genus one summands and denote them by $\Sigma_2^\pm$ (see figure \ref{gen2-loops}).
	
	 Throughout this article, we assume that $c_Q$ intersects the curves $A, B, C, X, Y, Z$ minimally and transversely. Since a simple closed curve on a thrice-boundered sphere either bounds a disk or is boundary parallel, the essential, simple, closed curve $c_Q$ has to intersect at least one of $A, B$ or $C$. $A \cup B \cup C$ separates $c_Q$ into essential, proper, simple arcs with endpoints on $A,B$ and $C$. Since every such arc requires exactly two endpoints, the total number of such arcs on both $\Sigma_2'$ and $\Sigma_2''$ are equal. We classify such arcs of $c_Q$ on the thrice boundered spheres $\Sigma_2'$ and $\Sigma_2''$ as unordered pairs of the following types: $(a,a), (b,b), (c,c), (a,b), (b,c), (a,c)$, where symbols $a,b,c$ represent any point of intersection of $c_Q$ with $A,B$ and $C$ respectively. For example the unordered pair $(a,b)$ denotes an arc with ends on $A$ and $B$. Further, throughout this article, when we write an arc of a certain type, eg. $(a,b)$ type, we always mean an essential, proper simple arc of that type.
	
	Because $c_Q$ is simple, not all arc-types can co-exist on $\Sigma_2'$ and $\Sigma_2''$  Table \ref{arc-table} presents such restrictions.
	\begin{table}[h]
		\caption{Arcs with intersecting counterparts}\label{arc-table}
		\centering\begin{tabular}{|c|c|}\hline
			If exists &  Ones that cannot exist\\ \hline
			\fbox{$(a,a)$} & \fbox{$(b,b), ~(b,c),~ (c,c)$} \\ \hline
			\fbox{$(b,b)$} & \fbox{$(a,a), (c,c)$}, $(a,c)$ \\ \hline
			\fbox{$(c,c)$} & \fbox{$(a,a), (b,b)$}, $(a,b)$ \\ \hline
			$(a,b)$ & $(c,c)$\\ \hline
			$(a,c)$ & $(b,b)$\\ \hline
			\fbox{$(b,c)$} & \fbox{$(a,a)$}\\ \hline
		\end{tabular}
	\end{table}	
	


\noindent

	As in \cite{akbas}, in a genus one summand $\Sigma_2^{\pm}$, an arc of slope $0$ is referred to as a meridional arc and that of slope $\infty$ is termed as a longitudinal arc.
	By $T_\omega$ we denote the Dehn twist (refer \cite{farb}) about a standard non-separating curve $\omega$ on $\Sigma_2$.	
	Throughout this article, we follow the standard convention of function composition while writing the word for an automorphism in $\mathcal{H}_2$. For example $T_\omega T_\theta$ means we apply $T_\theta$ first and then $T_\omega$. 
	



\section{The elements in $G_2$}
	A set $S=\{\alpha,\beta,\gamma,\delta\}$ of generators of $\mathcal{H}_2$ has been described in \cite{scharlemann2003automorphisms}. $\alpha$ represents the involution of $\Sigma_2$, $\gamma$ captures the rotational symmetry of $\Sigma_2$ and  $\beta$ represents the half-twists about the standard reducing curve $c_P$ (see figure \ref{generate2}).
	\begin{figure}[h]
	\centering\includegraphics[width=\linewidth]{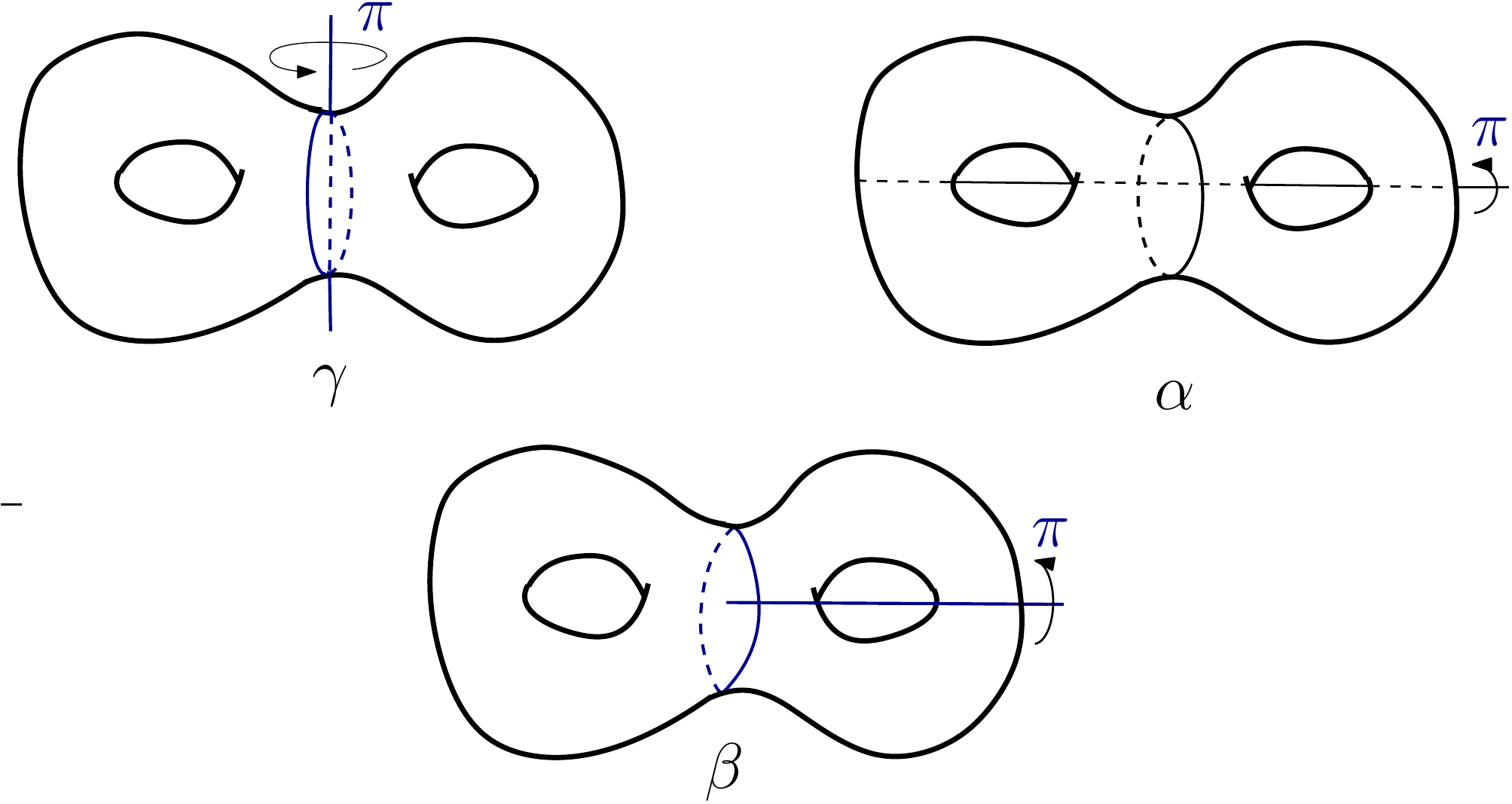}
	\caption{The automorphisms $\alpha,\gamma,\beta$ in $\mathcal{H}_2$}\label{generate2}
	\end{figure} 
	$\delta$ is an order $3$ automorphism as shown in figure \ref{generate2a}.  The automorphisms $\alpha,\beta$ and $\gamma$ keep the standard sphere invariant.
	
	Computations are easier using Dehn twists about non-separating curves, which generate the $\mathcal{MCG}(\Sigma_2)$, and so expressing automorphisms in $\mathcal{H}_2$ using these Dehn twists have a computational advantage.  With this in view, we replace $\gamma$ by an order two rotation $\nu$ (figure \ref{generate2a}) and also replace $\delta$ by $\varphi$.  We describe $\beta$ and $\varphi$ in terms of Dehn-twists about certain non-separating closed curves on $\Sigma_2$ so that we have a computationally simpler set of elements $G_2 = \{\alpha,\beta,\nu,\varphi\}$. We show that $G_2$ generates $\mathcal{H}_2$. We also write $\delta$ and $\gamma$ as words in the alphabet of $G_2$.

\begin{figure}[h]
\vspace*{8pt}
	\begin{minipage}{.49\linewidth}
	\centering\includegraphics[width=.7\linewidth]{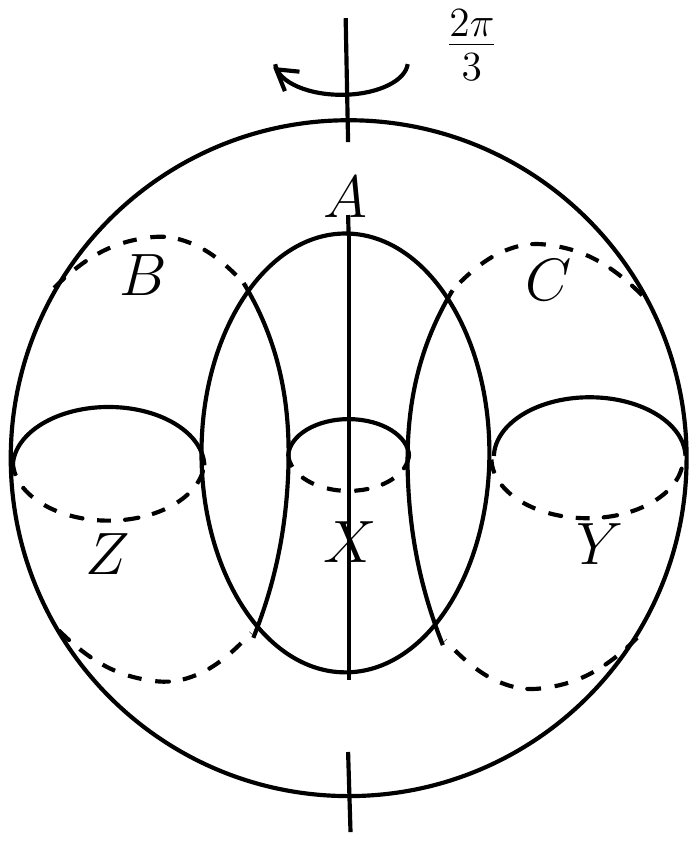}
	\end{minipage}
	\begin{minipage}{.49\linewidth}
	\centering\includegraphics[width=\linewidth]{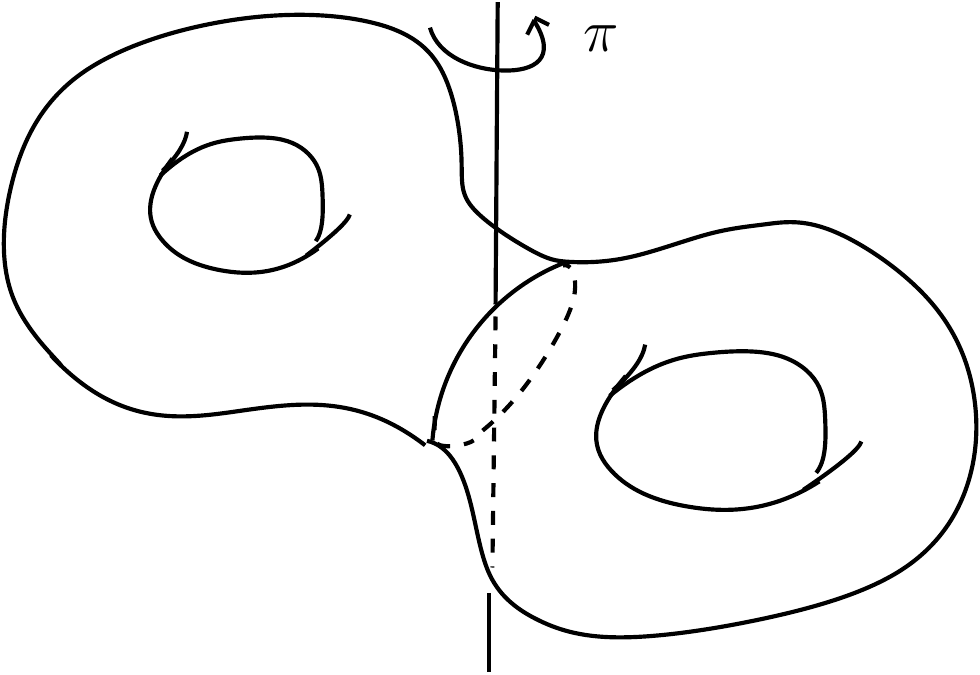}
	\end{minipage}
	\caption{Automorphisms $\delta$ and $\nu$}\label{generate2a}
\end{figure}	
		
\subsection{Automorphisms $\alpha$ and $\nu$}
From the description of $\alpha$ and $\nu$ it follows that
$$\alpha(A)=A, \alpha(B)=B, \alpha(C)=C, \alpha(\Sigma^\pm)=\Sigma^\pm, \alpha(\Sigma')=\Sigma'' \text{ and } \alpha(\Sigma'')=\Sigma'$$
and $$\nu(A)=A, \nu(B)=C, \nu(C)=B, \nu(\Sigma^\pm)=\Sigma^\mp, \nu(\Sigma')=\Sigma' \text{ and } \nu(\Sigma'')=\Sigma''.$$
But $$\gamma(A)=A, \gamma(B)=C, \gamma(C)=B, \gamma(\Sigma^\pm)=\Sigma^\mp, \gamma(\Sigma')=\Sigma'' \text{ and } \gamma(\Sigma'')=\Sigma'.$$
We note that $\gamma$ swaps $\Sigma_2'$ and $\Sigma_2''$ whereas $\nu$ leaves them invariant. 
\begin{lemma}[Properties of $\alpha$ and $\nu$]
	\begin{enumerate}[(i)]
		\item $\alpha^2 = \nu^2 = 1$.
		\item $\alpha\nu=\nu\alpha = \gamma$.
	\end{enumerate}
\end{lemma}
\begin{proof}
	The first part is immediate from the description of $\alpha$ and $\nu$. For the second part, figure \ref{gamma-nu} shows that $\gamma^{-1}\alpha\nu$ fixes the curves $A,B,C,X,Y,Z$ and also preserves $\Sigma_2'$ and $\Sigma_2''$. Therefore $\gamma^{-1}\alpha\nu$ is identity. Now $\gamma^2 = (\alpha\nu)^2 = 1 = \alpha^2\nu^2$ by $(i)$. So $\alpha\nu=\nu\alpha = \gamma$.
	\begin{figure}
	\centering\includegraphics[width=\linewidth]{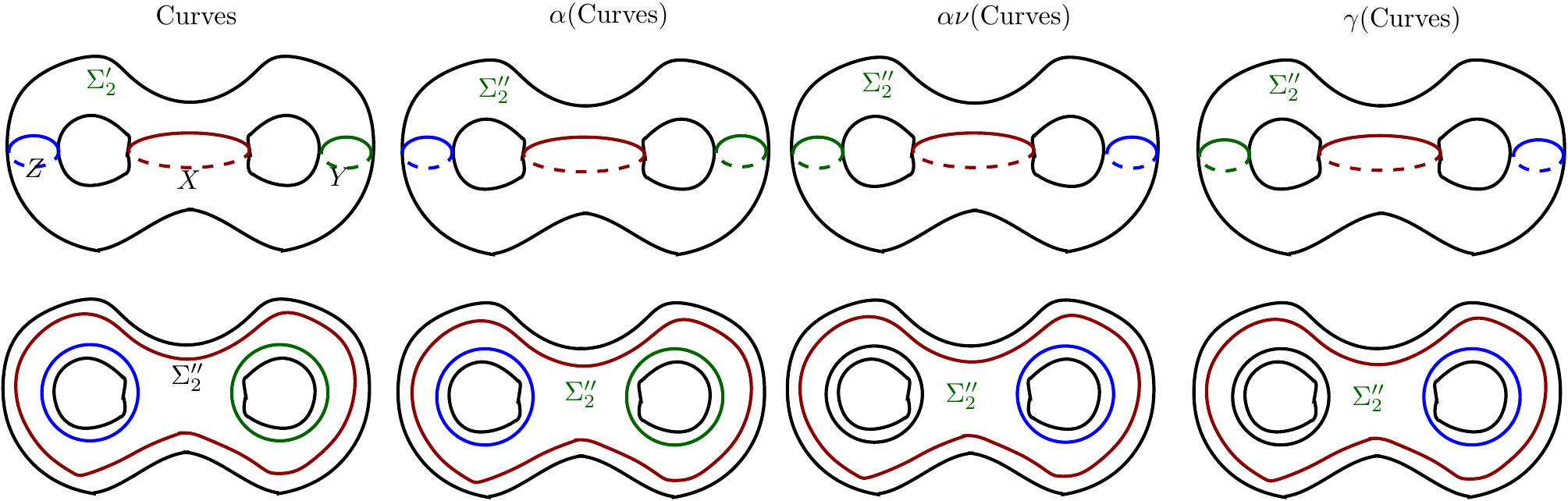}
	\caption{$\alpha\nu = \gamma$.}\label{gamma-nu}
	\end{figure}
\end{proof}

\subsection{Automorphism $\beta$}
	
	$\beta$ is a half twist about the standard reducing curve $c_P$. Using Dehn twists about $Y$ and $C$ (figure \ref{gen2-loops}), we can express $\beta$ as
	\[\beta = (T_{C}T_{Y})^3 = T_{C}T_{Y}T_{C}T_{Y}T_{C}T_{Y},\]
	This word-presentation is not unique. For example, using the braid relation, we can also express $\beta$ as
	\[\beta = (T_CT_YT_C)^2 = (T_YT_CT_Y)^2.\]
	Figure \ref{betaX} illustrates the computations of the application of $\beta$ on $A$ and $X$. Note that $(T_CT_YT_C)$ exchanges $Y$ and $C$. So they are invariant under $\beta$.
	\begin{figure}[h]
		\centering\centering\includegraphics[width=\linewidth]{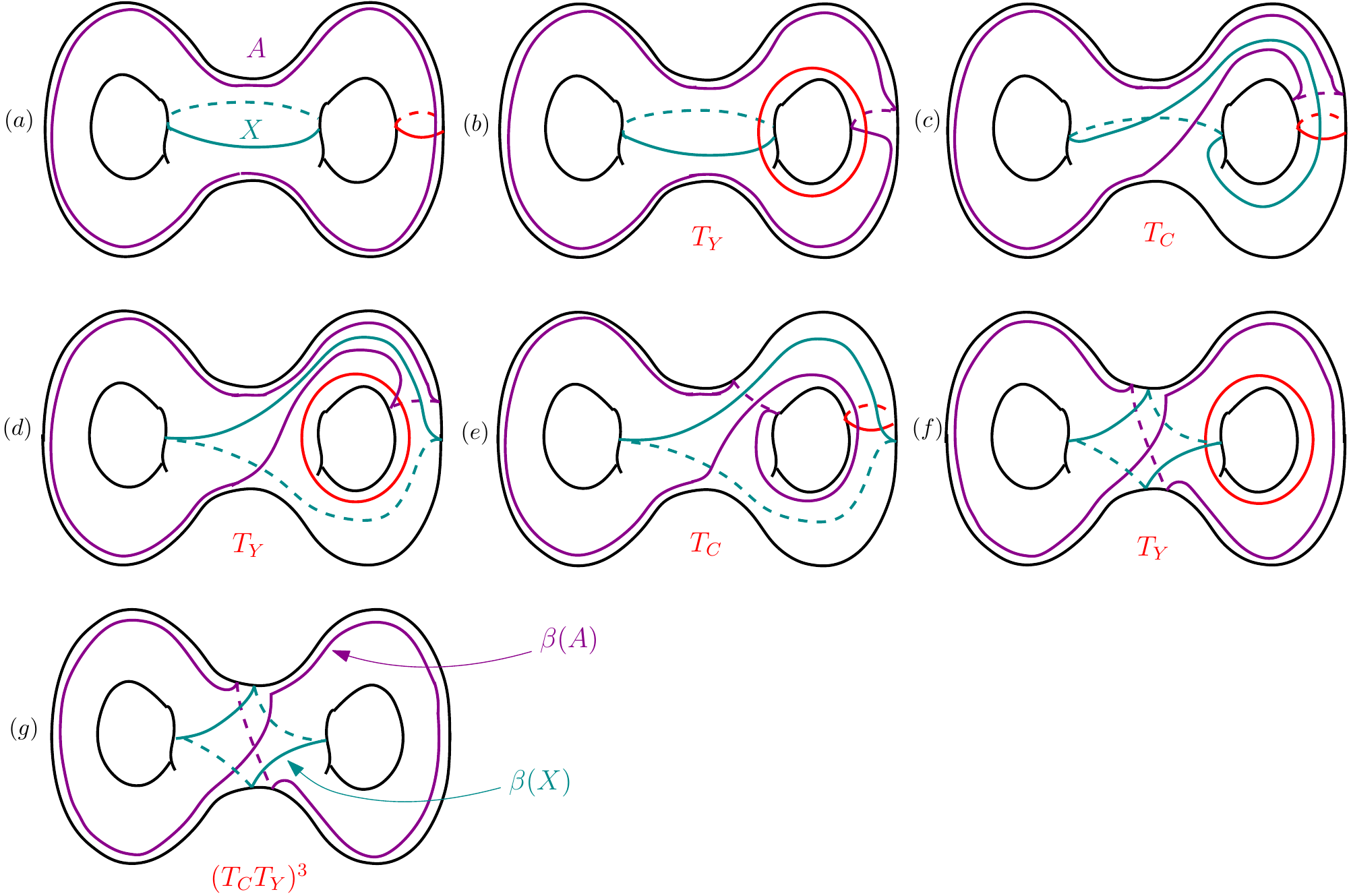}
		\caption{Computation of $\beta(X)$}\label{betaX}
	\end{figure}
	\begin{figure}[h]
		\centering\centering\includegraphics[width=.8\linewidth]{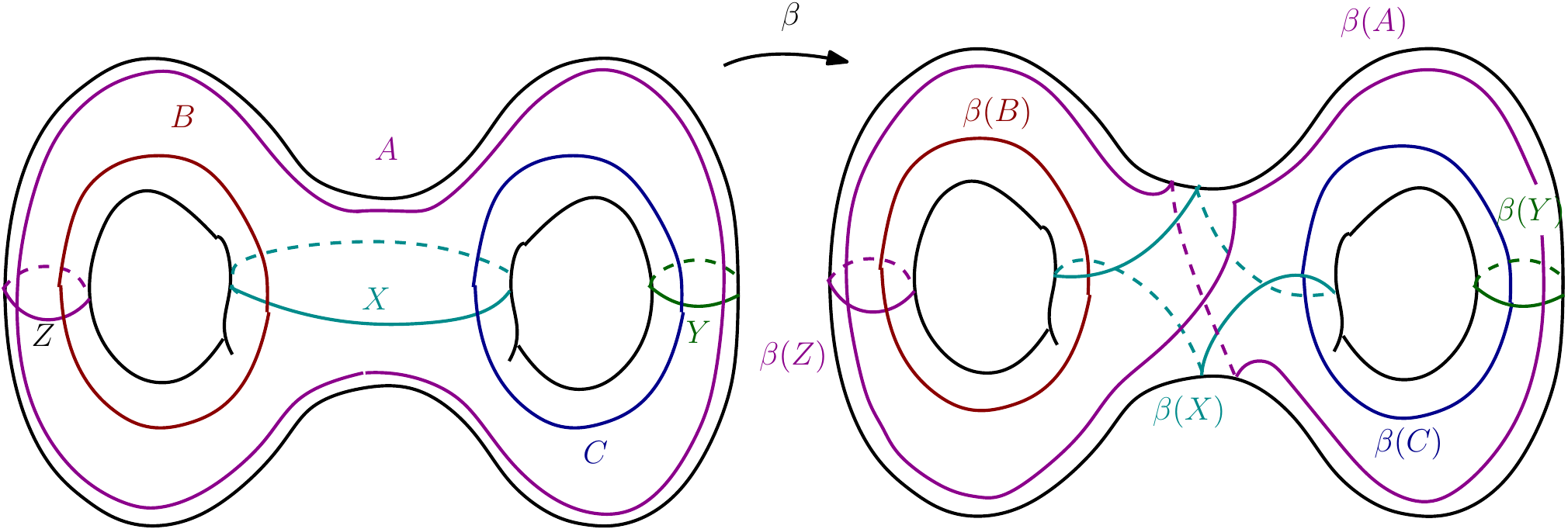}
		\caption{Action of $\beta$ on $\Sigma_2$}\label{beta2-act}
	\end{figure}
	Since this $\beta^{-1}$ composed with the half-twist discussed in \cite{scharlemann2003automorphisms} fixes all the essential non-separating loops $A,B,C,X,Y$ and $Z$ along with $\Sigma_2'$ and $\Sigma_2''$, the composition is identity on $\Sigma_2$. Therefore $\beta\in\mathcal{H}_2$ and is indeed the half-twist.

	Now from figure \ref{beta2-act}, one can observe that $\beta$ leaves $\Sigma_2^\pm$ invariant and only increases or reduces the intersection of $c_Q$ with $A$ in a collar neighbourhood of $c_P$ and at the same time introduces or gets rid of arcs in that region.
	
	\begin{lemma}\label{beta-prop}
		$\beta$ exhibits the following properties:
		\begin{enumerate}[(i)]
			\item Order of $\beta$ is infinite.
			\item $\beta$ commutes with $\alpha$ and $\nu$.
		\end{enumerate}
	\end{lemma}
	\begin{proof}
		The first part naturally follows from the fact that $\beta^{n+1}(X)\cdot A>\beta^{n}(X)\cdot A$, for all $n\in \mathbb{N}$.
		
		For the second part, it is easy to verify that both $\nu\beta\nu\beta^{-1}$ and $\alpha\beta\alpha\beta^{-1}$ fix $A,B,C,X,Y$ and $Z$ along with $\Sigma_2'$ and $\Sigma_2''$. Therefore, both are identity in $\mathcal{MCG}(\Sigma_2)$ and so the result follows.
	\end{proof}


\subsection{Automorphism $\varphi$} 
	$\varphi$ can be described as
	$$\varphi = T_{Z}^{-1}T_{Y}T_{C}T_{Y}T_XT_{C}T_{Y}.$$
	The effect of $\varphi$ on the standard loops on $\Sigma_2$ is shown in figure \ref{phi2-act}.
	\begin{figure}[h!]
		\centering\includegraphics[width=\linewidth]{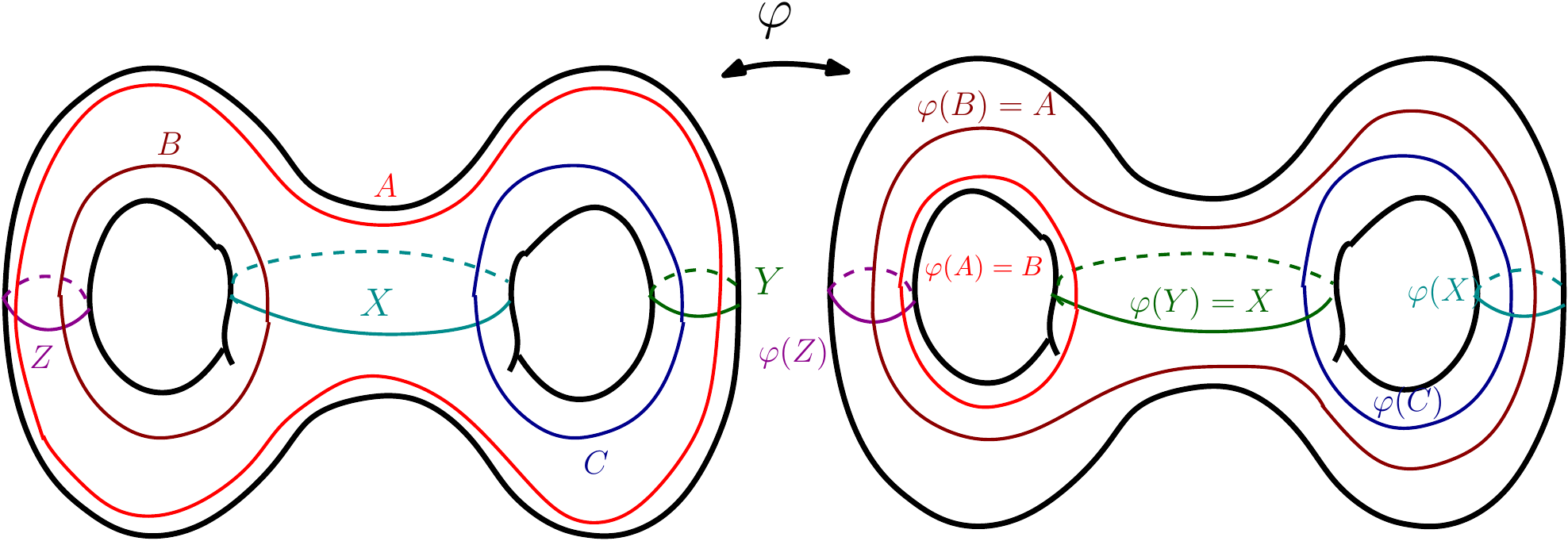}
		\caption{Action of $\varphi$ on $\Sigma_2$}\label{phi2-act}
	\end{figure} 
	
	$\varphi$ exchanges the loops $Y$ and $X$ but leaves $C$ and $Z$ invariant. From the action of $\varphi$ on curves in figure \ref{phi2-act}, we give the following lemma:
	
	\begin{lemma}\label{phi-prop}
	The automorphism $\varphi$ satisfies the following:
	\begin{enumerate}[(i)]		
		\item $\varphi(A)=B$, $\varphi(B)=A$, $\varphi(C)=C$, $\varphi(X)=Y$, $\varphi(Y)=X$, $\varphi(Z)=Z$, $\varphi(\Sigma_2')=\Sigma_2'$ and $\varphi(\Sigma_2'')=\Sigma_2''$. So, $\varphi^2=1$.
		\item $\varphi \in \mathcal{H}_2$.
		\item $\varphi$ commutes with $\alpha$.
	\end{enumerate}
	\end{lemma}

\begin{proof}
	\begin{enumerate}[(i)]
		\item Figure \ref{phi-comp} demonstrates the verification of $\varphi(A)=B$.
		\begin{figure}[h]
			\centering\includegraphics[width=14cm]{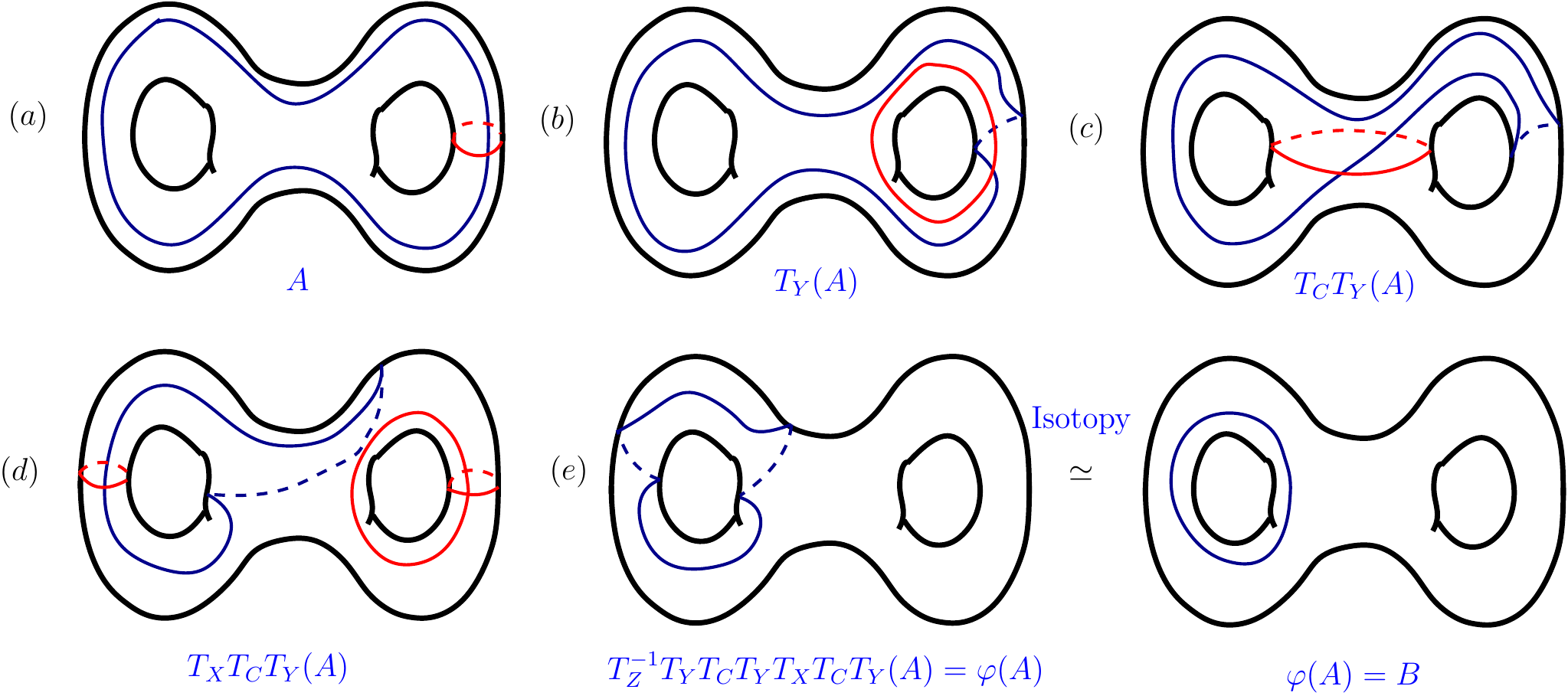}
			\caption{Computation of $\varphi(A)$}\label{phi-comp}
		\end{figure}
	By computing in a similar manner, one can verify that the first result follows from figure \ref{phi2-act}. From this it follows that $\varphi^2$ fixes  $A,B,C,X,Y,Z, \Sigma_2'$ and $\Sigma_2''$ (refer figure \ref{gen2-loops}) on $\Sigma_2$. Therefore $\varphi^2\simeq 1$. 
		\item Now consider the eyeglass move $\varphi_\theta$ in \cite{zupan2019powell} for $\Sigma_2$.
		\begin{figure}[h]
			\centering\includegraphics[width=12cm]{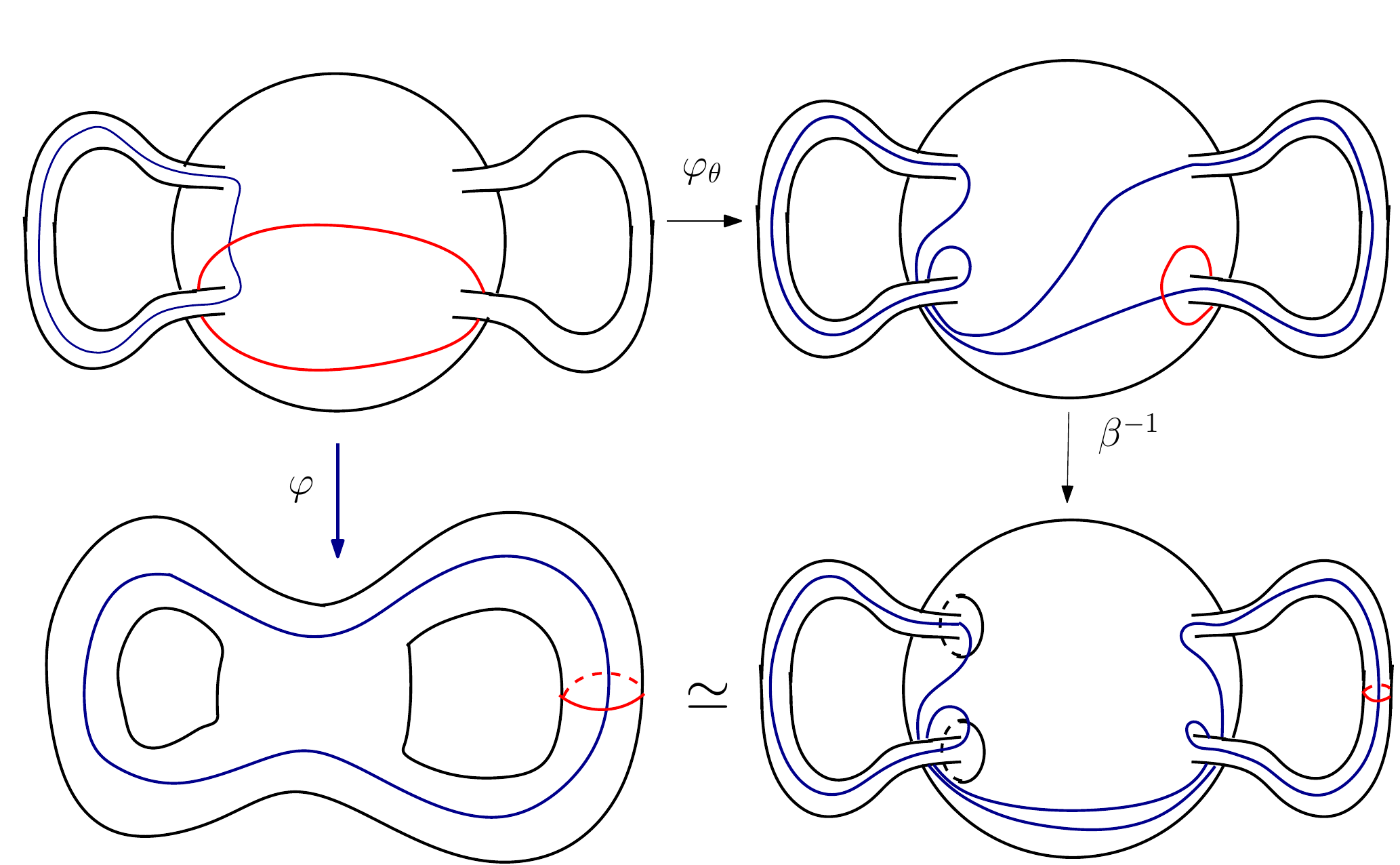}
			\centering\includegraphics[width=12cm]{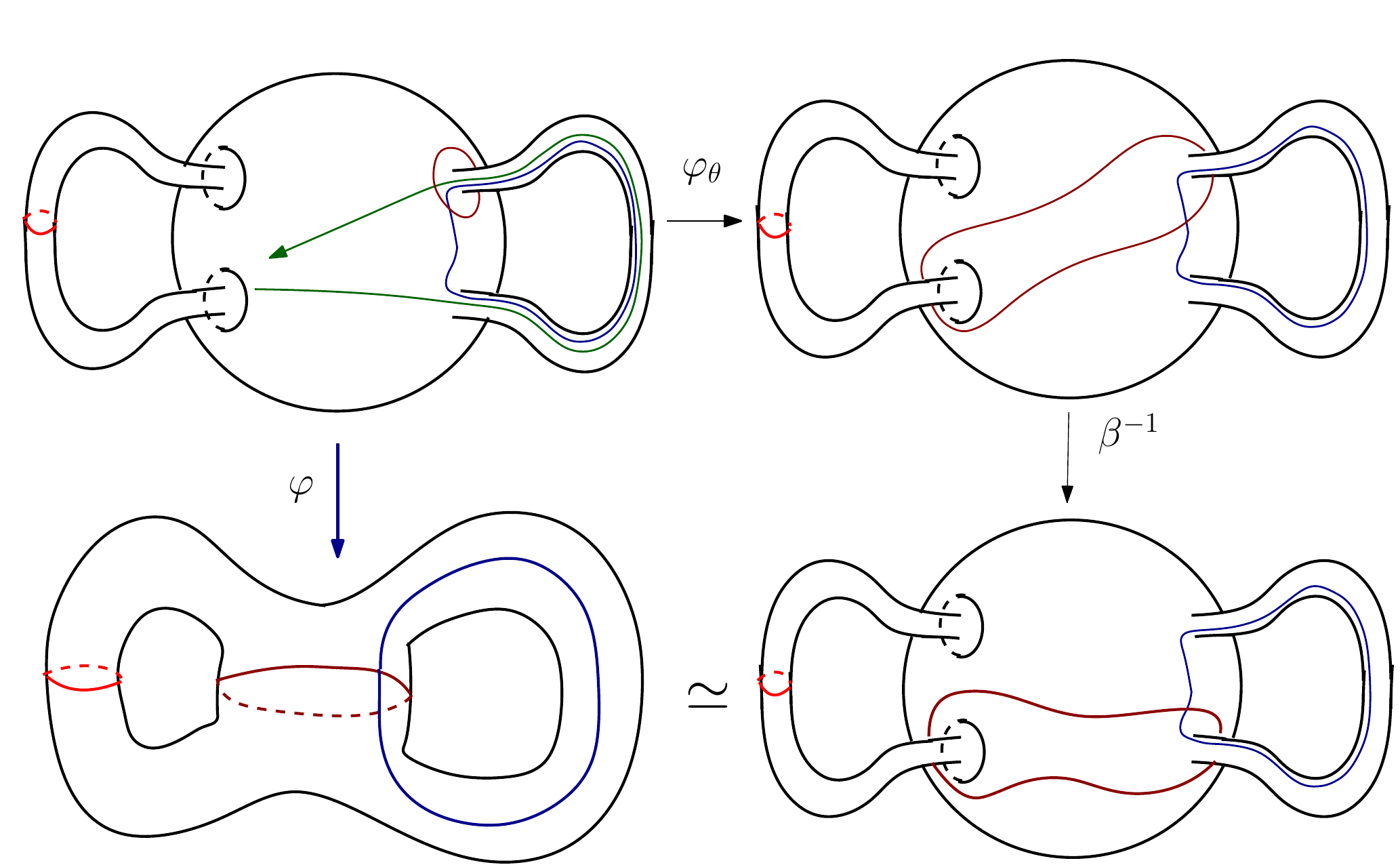}
			\caption{Automorphism $\varphi$ comparision with eyeglass move }\label{eye-phi}
		\end{figure}
		Figure \ref{eye-phi} demonstrates that $\varphi^{-1}(\beta^{-1}\varphi_\theta)$ fixes $A,B,C,X,Y,Z, \Sigma_2'$ and $\Sigma_2''$ (refer figure \ref{gen2-loops}) on $\Sigma_2$. Therefore, $\varphi^{-1}(\beta^{-1}\varphi_\theta)$ is isotopic to identity on $\Sigma_2$. This implies $\beta^{-1}{\varphi_\theta} = \varphi$ on $\Sigma_2$. Hence the result. 
		\item This is clear from (i).
	\end{enumerate}
\end{proof}

Since $\varphi$ keeps $\Sigma_2'$ and $\Sigma_2''$ invariant, we can discuss the action of $\varphi$ on the arc types mentioned in Table \ref{arc-table}. The schematic in figure \ref{phi2-schem} illustrates the action of $\varphi$ on these arc types.
 	\begin{figure}[h!]
 		\centering\centering\includegraphics[width=\linewidth]{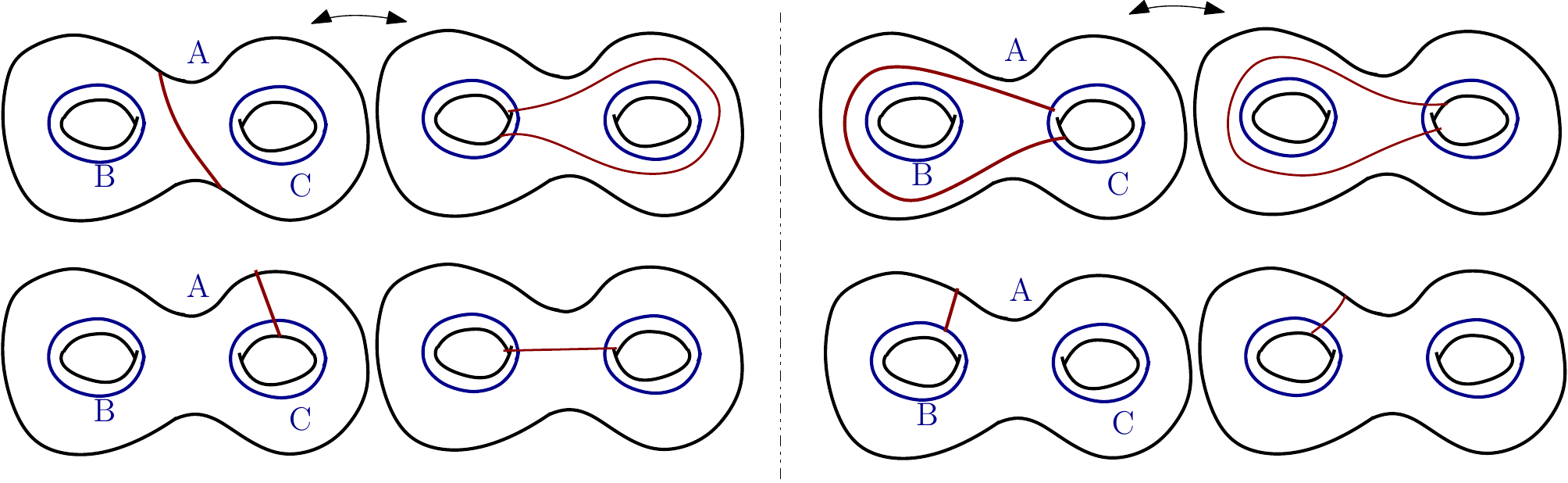}
 		\caption{Schematic presentation of action of $\varphi$ on arcs in $\Sigma_2'$}\label{phi2-schem}
 	\end{figure} 
 	

\section{Complexity and Reduction algorithm}

In this section we give an algorithm that will provide us a word from the alphabet of the set $G_2$ that maps an arbitrary reducing sphere to the standard one. 
Consider the setup described in section \ref{setup} of the genus two Heegaard splitting of $S^3$.
Let $Q$ be a reducing sphere and let $c_Q$ be the corresponding reducing curve on $\Sigma_2$. Denote $c_Q\cdot A$,  $c_Q\cdot B$ and $c_Q\cdot C$ by $n_{AQ}$, $n_{BQ}$ and $n_{CQ}$ respectively. Since $c_Q$ is a separating closed curve on $\Sigma_2$, $n_{AQ}$, $n_{BQ}$ and $n_{CQ}$  must be even numbers.
First, for $Q$, we introduce a measure $\mathcal{C}(Q)$ defined as
	\[\mathcal{C}(Q) = \frac{1}{2} (n_{AQ}) + n_{BQ} + n_{CQ}\]
The following is the motivation for this measure.
\begin{lemma}
	$\mathcal{C}(Q)=1$ if and only if $Q=P$.
\end{lemma}
\begin{proof}
If $Q=P$, then $\mathcal{C}(Q)=\mathcal{C}(P)=1$.\\
Conversely, $\mathcal{C}(Q)=1$ only if $n_{AQ}=2$ and $n_{BQ}=n_{CQ}=0$. Therefore $c_Q\cap\Sigma_2'$ (similarly $c_Q\cap\Sigma_2''$) is a single $(a,a)$ arc. Therefore upto isotopy, $c_Q=c_P$ and hence the result.
\end{proof}
%
	\begin{lemma}\label{arc-lemma}
		Let $c_Q$ be any reducing curve on $\Sigma_2$. Then $c_Q$ must contain atleast one essential proper simple arc of the type $(a,a), (b,b), (c,c)$ or $(b,c)$.
	\end{lemma}
	
	\begin{proof}
		All arcs of $c_Q$ on $\Sigma_2'$ or $\Sigma_2''$ here are assumed to be essential proper and simple. If possible, let $c_Q$ not have any $(a,a),(b,b),(c,c)$ and $(b,c)$ arc. Then all arcs in $c_Q\cap\Sigma_2'$ (similarly in $c_Q\cap\Sigma_2'$) are either $(a,b)$ or $(a,c)$ arcs.
		So every arc with one end on $B$ (similarly on $C$) must have the other end on $A$ on both $\Sigma_2'$ and $\Sigma_2''$ and so if $n_{BQ} \neq 0, n_{CQ} \neq 0$ then $c_Q$ can not have an arc of slope $\infty$ in any of $\Sigma_2^\pm$. This contradicts Lemma 4 of \cite{scharlemann2003automorphisms} as $Q$ is non-standard. So, either $n_{BQ}=0$ or $n_{CQ}=0$.
		
		Without loss of generality, let $n_{CQ}=0$. Then all arcs in $c_Q\cap\Sigma_2'$ (resp. $\Sigma_2''$) are $(a,b)$-arcs. Then there exists a pair of $(b,b)$-arcs say $\eta',\eta''$ respectively in $\Sigma_2'$ and $\Sigma_2''$ with common end-points such that $c_Q$ lies in one of the components of $T'=\Sigma_2-(\eta'\cup \eta'')$. We perform surgery of $T'$ along $\eta'\cup \eta''$ to obtain the torus $T$ on which $c_Q$ lies and must separate $T$. Hence $c_Q$ must bound a disk on $T$. Just by starting at any point on $c_Q$ and following the curve, we can see that any orientation on $c_Q$ will induce an orientation on $(a,b)$ arcs in $\Sigma_2'$ either as all arcs starting on $A$ and ending on $B$ or as all arcs starting on $B$ and ending on $A$. But then, the algebraic intersection number of $c_Q$ with $A$ (or $B$) is not zero, a contradiction for $c_Q$ to bound a disk.
		
		Therefore, $c_Q$ must contain atleast one of $(a,a),(b,b),(c,c)$ and $(b,c)$ arc.
	\end{proof}
	
\begin{lemma}
	Let $Q\neq P$ be any reducing sphere. Then $n_{AQ}\ne n_{BQ}+n_{CQ}$.
\end{lemma}
\begin{proof}
	Suppose that $n_{AQ} = n_{BQ}+n_{CQ}$. All arcs of $c_Q$ on $\Sigma_2'$ or $\Sigma_2''$ have to be essential and simple due to the minimal position of $c_Q$ with respect to $A, B, C, X, Y$ and $Z$. So if there is an essential simple $(b,b), (c,c)$ or a $(b,c)$ arc of $c_Q$ on $\Sigma_2'$ or $\Sigma_2''$, then by pairing the remaining points, its easy to see that there has to be an essential simple $(a,a)$ arc. Such an $(a,a)$ arc must be the outermost and must also allow for the presence of the other arcs. This implies that such an $(a,a)$ arc and the curve $A$ bound a bigon, contradicting the minimal position of $c_Q$ with respect to $A$. So none of these arcs of $c_Q$ on $\Sigma_2'$ or $\Sigma_2''$ is an $(a,a),(b,b),(c,c)$ or $(b,c)$ arc and all arcs  must be $(a,b)$ and $(a,c)$ arcs. But this contradicts the lemma \ref{arc-lemma} that $c_Q$ must have atleast one of $(a,a),(b,b),(c,c)$ or $(b,c)$ arcs. So, $n_{AQ}\ne n_{BQ}+n_{CQ}$.
\end{proof}

\begin{lemma}
	Let $Q\ne P$ be any reducing sphere. Then $n_{BQ} \ne n_{CQ}$.
\end{lemma}
\begin{proof}
	Lemma 1 in \cite{akbas} showed that $N(Q,\Sigma_2^-,0)\ne N(Q,\Sigma_2^-,\infty)$. If $N(Q,\Sigma_2^-,0)\neq 0$ and if $N(Q,\Sigma_2^-,a)\ne 0$ for some $a\in\mathbb{Q}$ then $a=0$ or $\frac{1}{p}$ for some $p\in \mathbb{N}$. Again from lemma 1 in \cite{akbas} we have $N(Q,\Sigma_2^-,a)=N(Q,\Sigma_2^+,\frac{1}{a})$. Therefore, 
	$n_{BQ} = N(Q,\Sigma_2^-,0)+pN\left(Q,\Sigma_2^-,\frac{1}{p}\right)$, and $n_{CQ} = N(Q,\Sigma_2^-,\infty)+N\left(Q,\Sigma_2^-,\frac{1}{p}\right)$.
	If $p=1$ then $n_{BQ} - n_{CQ} = N(Q,\Sigma_2^-,0) - N(Q,\Sigma_2^-,\infty)\ne 0$. Now $p\ne 1$ implies $N(Q,\Sigma_2^-,\infty) = N(Q,\Sigma_2^+,0)=0$. Therefore 
	\begin{eqnarray*}
		n_{BQ} - n_{CQ} &=& N(Q,\Sigma_2^-,0) + pN\left(Q,\Sigma_2^-,\frac{1}{p}\right) - N\left(Q,\Sigma_2^+,p\right)\\ 
		&=& N(Q,\Sigma_2^-,0) + (p-1)N\left(Q,\Sigma_2^-,\frac{1}{p}\right) \ne 0
	\end{eqnarray*}
 Therefore $Q\ne P \implies n_{BQ} \ne n_{CQ}$.
\end{proof}

\begin{theorem}
	Let $Q\ne P$ be any reducing sphere. If $n_{AQ} > n_{BQ} + n_{CQ}$ then exactly one of the following occurs:\\ (i) $\mathcal{C}(\beta(Q))<\mathcal{C}(Q)$ and $\mathcal{C}(\beta^{-1}(Q))>\mathcal{C}(Q)$\\
(ii)	$\mathcal{C}(\beta(Q))>\mathcal{C}(Q)$ and $\mathcal{C}(\beta^{-1}(Q))<\mathcal{C}(Q)$.\\
In any case if $n_{AQ} > n_{BQ} + n_{CQ}$, $\mathcal{C}(\varphi(Q))>\mathcal{C}(Q)$ and $\mathcal{C}(\varphi\nu(Q))>\mathcal{C}(Q)$.
\end{theorem}

\begin{proof}
 	Let us consider a collar neighbourhood $\mathcal{A}$ of $c_P$ on $\Sigma_2$. 
 	$c_Q$ is assumed to be in minimal position with respect to $X$ and $\partial\mathcal{A}$ along with $A,B,C$. Since ${n_A}_Q > {n_B}_Q + {n_C}_Q$, $c_Q$ contains an $(a,a)$ arc on both $\Sigma_2'$ and $\Sigma_2''$. Let us denote the pair of arcs in $\mathcal{A}\cap A$ and $\mathcal{A}\cap X$ by $a',a''$ and $x',x''$ respectively. All the $(a,a)$ arcs on $\Sigma_2'$ (similarly on $\Sigma_2''$) are disjoint parallel arcs intersecting $X$ exactly once (on $x'$ or $x''$). Clearly the ends of $(a,a)$ arcs on $\Sigma_2'$ (and similarly on $\Sigma_2''$) can be nested around some meridian in $\Sigma_2^\pm$. We pick the innermost and the outermost $(a,a)$ arc. All $(a,a)$ arcs must lie parallel between them and all $(c,a)$ arcs in $\Sigma_2'$ (resp. on $\Sigma_2''$) and $(b,a)$ arcs in $\Sigma_2'$ (resp. on $\Sigma_2''$) must lie on the opposite sides of the $(a,a)$ arcs.  
 	\begin{figure}[h!]
 		\centering\includegraphics[width=.8\linewidth]{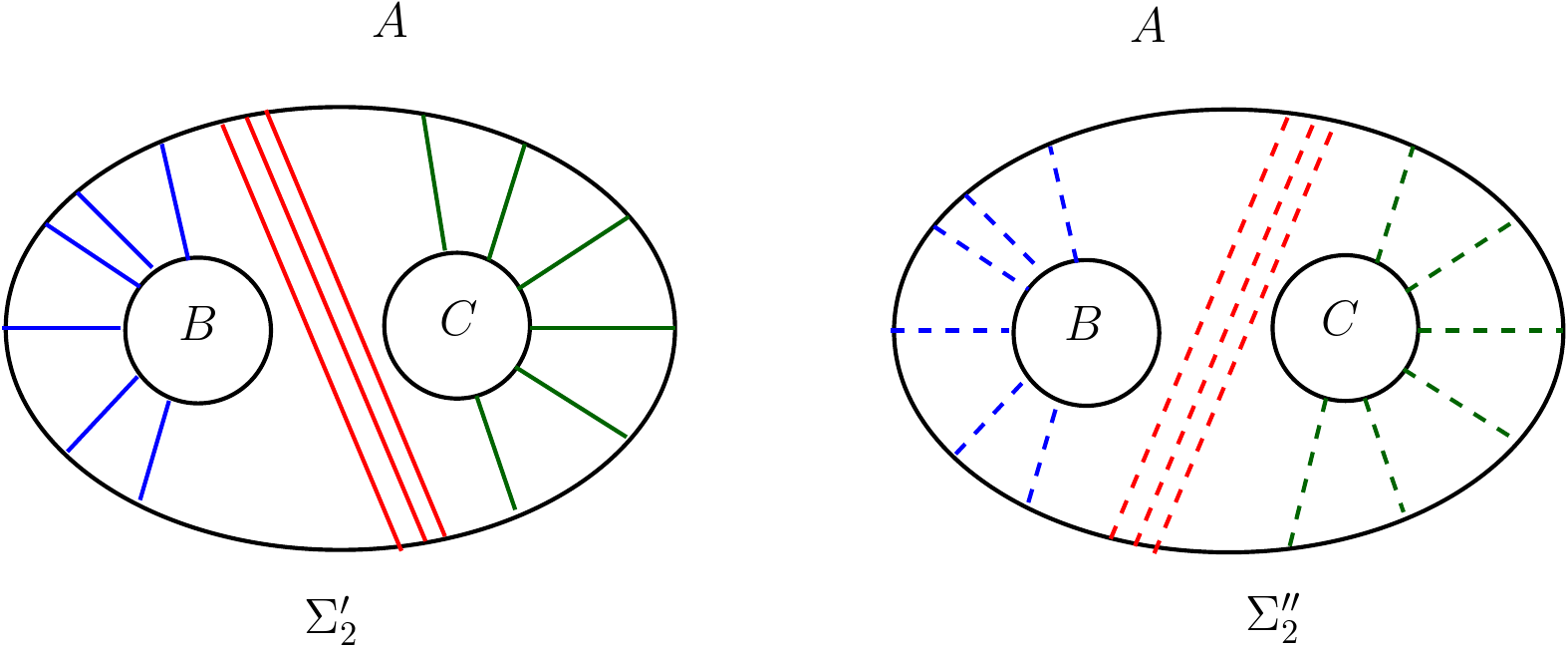}
 		\caption{$(a,a),(a,b),(a,c)$ appear in groups}\label{nested-arcs}
 	\end{figure}
 Therefore all $(b,a)$ arcs are nested on $\Sigma_2^-$ (or $\Sigma_2^+$) and all $(c,a)$ arcs are nested on $\Sigma_2^+$ (or $\Sigma_2^-$). Let us isotope $c_Q$ on $\Sigma_2$ such that the outermost $(a,b)$ and $(a,c)$ arcs have their $a$-ends and their intersection with $X$ (if any) inside $\mathcal{A}$.
 \begin{figure}[h!]
 	\centering\includegraphics[width=.8\linewidth]{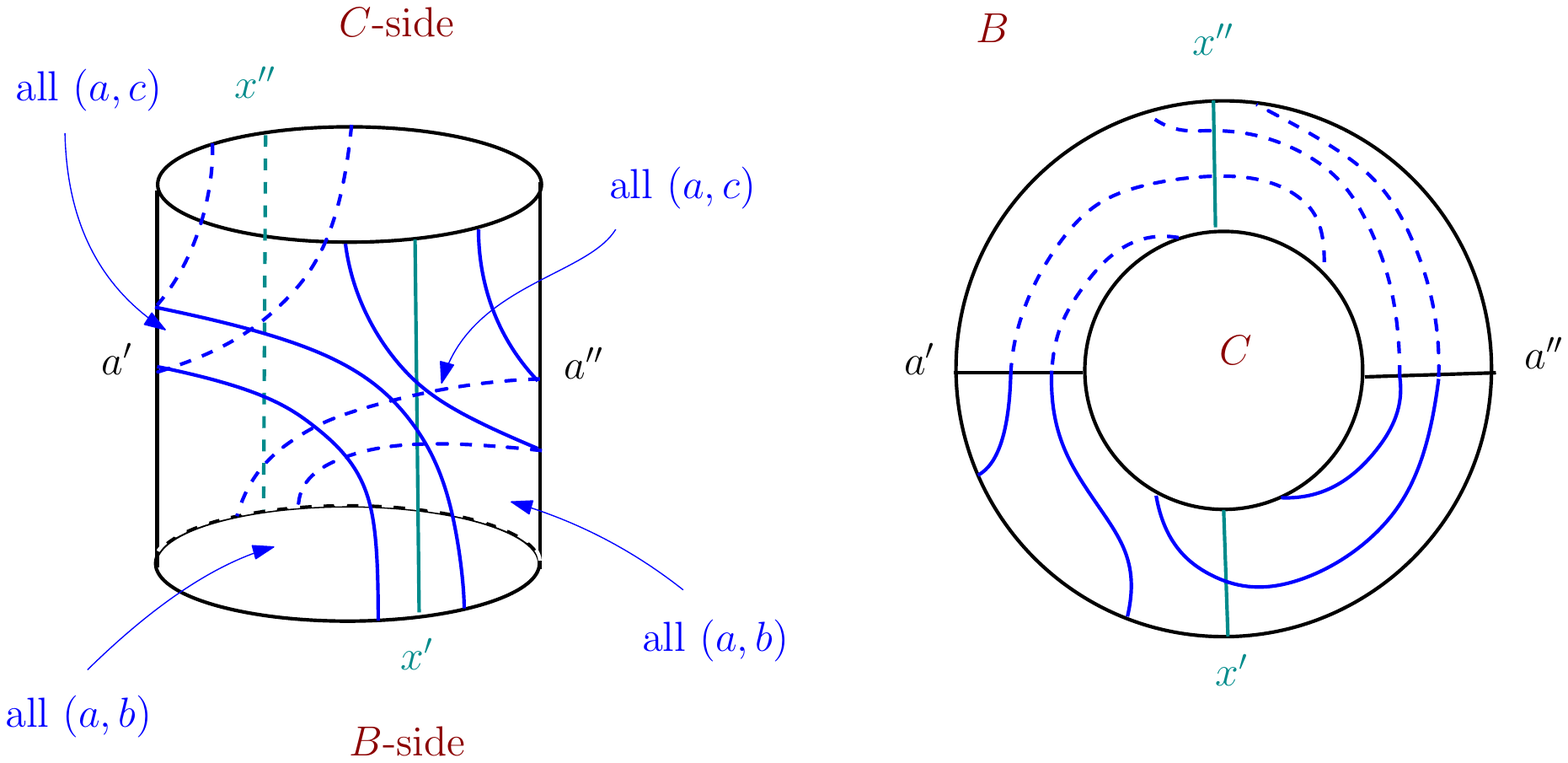}
 	\caption{Collar neighbourhood of $c_P$; annulus $\mathcal{A}$}\label{annulus}
 \end{figure}
 With this setup, it is easy to observe that the innermost $(a,a)$ arcs on both $\Sigma_2'$ and $\Sigma_2''$ ensures that every arc on $\mathcal{A}$ must intersect $a'$ or $a''$ atleast once. Now from figure \ref{annulus} it can be observed that suitable application of $\beta$ or $\beta^{-1}$ reduces $n_{AQ}$ by one for each arc connecting one boundary of  $\mathcal{A}$ to the other. Correspondingly either $\beta^{-1}$ or $\beta$  increases $n_{AQ}$ by the same amount. Moreover, application of $\varphi$ (or $\varphi\nu$) too increases $\mathcal{C}(Q)$. As $\varphi$ only exchanges $n_{AQ}$ and $n_{BQ}$ keeping $n_{CQ}$ unchanged. So $\varphi$ increases $\mathcal{C}(Q)$. Therefore, the reduction in $\mathcal{C}(Q)$ is done only by a unique choice between $\beta$ or $\beta^{-1}$.

 Moreover, the amount of reduction is equal to the number of non-trivial arcs of $c_Q$ on the annulus $\mathcal{A}$. That number can be represented as
 \[ N_\beta = 2\left[{n_B}_Q + {n_C}_Q - N\left(Q,\Sigma_2^-,\frac{1}{p}\right)\right], \quad p\geq 1.\]
 \end{proof}
		
\begin{theorem}\label{phi-thm}
	Let $Q$ be any non-standard reducing sphere and $\varphi(Q)=R$ and $\varphi\nu(Q)=S$. Then
\begin{enumerate}[(i)]
	\item if $n_{AQ}<n_{BQ}+n_{CQ}$ and $n_{BQ}>n_{CQ} $ then $\mathcal{C}(R)<\mathcal{C}(Q)$ and $n_{AR} > n_{BR} + n_{CR}$;
	\item if $n_{AQ}<n_{BQ}+n_{CQ}$ and $n_{BQ}<n_{CQ} $ then $\mathcal{C}(S)<\mathcal{C}(Q)$ and $n_{AS} > n_{BS} + n_{CS}$
\end{enumerate}
In either case $\mathcal{C}(\beta(Q))>\mathcal{C}(Q)$ and $\mathcal{C}(\beta^{-1}(Q))>\mathcal{C}(Q)$
\end{theorem}
	
\begin{proof}
\begin{enumerate}[(i)]
%
	\item If $n_{AQ}<n_{BQ}+n_{CQ}$, then on $\Sigma_2'$ (or $\Sigma_2''$), $c_Q$ cannot contain an $(a,a)$ arc. So by Lemma \ref{arc-lemma}, it must contain a $(b,b), (c,c)$ or a $(b,c)$ arc. Further, if $n_{BQ}>n_{CQ}$, then all the arcs cannot be $(b,c)$ arcs. So there has to be a $(b,b)$ or a $(c,c)$. Once again since the presence of a $(b,b)$ and $(c,c)$ arcs is mutually exclusive and since $n_{BQ}>n_{CQ}$, $c_Q$ has to contain atleast one $(b,b)$ arc and no $(c,c)$ arcs. So if $n_{AQ}<n_{BQ}+n_{CQ}$ and $n_{BQ}>n_{CQ}$, we infer from Table \ref{arc-table} that the only possible arcs of $c_Q$ on $\Sigma_2'$ (or $\Sigma_2''$) are $(b,b), (b,c)$ and $(a,b)$ with atleast one $(b,b)$ arc. This implies that every such arc must have an end point on $B$. So $n_{BQ} > n_{CQ} + n_{AQ}$. Hence in this case,
	$$\mathcal{C}(R)-\mathcal{C}(Q)$$ 
	$$=\frac{1}{2}(n_{AR})+n_{BR}+n_{CR} - \left(\frac{1}{2}(n_{AQ})+n_{BQ}+n_{CQ}\right)$$
	$$= \frac{1}{2}(n_{BQ})+n_{AQ}+n_{CQ} - \left(\frac{1}{2}(n_{AQ})+n_{BQ}+n_{CQ}\right) $$
	$$= \frac{1}{2}(n_{AQ}) - \frac{1}{2}(n_{BQ}) <0,$$
	\item By a symmetric argument, if $n_{AQ}<n_{BQ}+n_{CQ}$ and $n_{BQ}<n_{CQ}$ then $n_{BQ} > n_{CQ} + n_{AQ}$. Hence in this case.
	$$\mathcal{C}(S)-\mathcal{C}(Q)$$ 
	$$= \frac{1}{2}(n_{AS})+n_{BS}+n_{CS} - \left(\frac{1}{2}(n_{AQ})+n_{BQ}+n_{CQ}\right)$$
	$$= \frac{1}{2}(n_{CQ})+n_{AQ}+n_{BQ} - \left(\frac{1}{2}(n_{AQ})+n_{BQ}+n_{CQ}\right)$$
	$$= \frac{1}{2}(n_{AQ}) - \frac{1}{2}(n_{CQ}) <0.$$
	This also proves that reduction of the complexity measure by $\varphi$ or $\varphi\nu$ is by $\frac{1}{2}\left(n_{AQ} - n_{BQ}\right)$ or $\frac{1}{2}\left(n_{AQ} - n_{CQ}\right)$ respectively. Further, we also note that $n_{AR}=n_{BQ} > n_{AQ} + n_{CQ} = n_{BR} + n_{CR}$ and likewise $n_{AS} > n_{BS} + n_{CS}$
	The assertion regarding application of $\beta$ or its inverse is easy to verify.
\end{enumerate}
\end{proof}

%
%
%

The above results lead us to the following algorithm.

\subsection{The algorithm}
	
	Based on the discussion presented above we can present the complexity reduction of an arbitrary reducing sphere of the genus two Heegaard splitting of $S^3$ via a finite step algorithm.

\noindent\begin{algorithm}
Let $Q$ be an arbitrary reducing sphere of genus two Heegaard splitting of $S^3$.
\begin{description}
		\item[Step-1] If $n_{BQ}=n_{CQ}=0$ then $Q$ is standard. Exit. Else go to step-2.
		\item[Step-2] While $n_{AQ}>n_{BQ}+n_{CQ}$, apply $\beta$ or $\beta^{-1}$ so that $\mathcal{C}(Q)$ decreases. Update $n_{AQ},n_{BQ}$ and $n_{CQ}$. Goto Step-3.
		\item[Step-3] If $n_{AQ}<n_{BQ}+n_{CQ}$ and $n_{BQ}>n_{CQ}$ apply $\varphi$ and update $n_{AQ},n_{BQ}$ and $n_{CQ}$, else if $n_{AQ}<n_{BQ}+n_{CQ}$ and $n_{CQ}>n_{BQ}$ apply  $\varphi\nu$ and update $n_{AQ},n_{BQ}$ and $n_{CQ}$. Go to Step-1. 
	\end{description} 
\end{algorithm}

As $\mathcal{C}(Q)$ decreases strictly at each step until $\mathcal{C}(Q)=1$, and since $n_{AQ}, n_{BQ}$ and $n_{CQ}$ are finite, the above algorithm terminates in finitely many steps. 
	
	Now since for any arbitrarily chosen reducing sphere $Q$, this algorithm provides an automorphism $f$ such that $f(Q)=P$ using only elements from $G_2$, and using the description of stabilizer of $P$ in \cite{scharlemann2003automorphisms} we conclude:
\begin{proposition}\label{prop-generate}
	$G_2$ generates $\mathcal{H}_2$.
\end{proposition}

\begin{theorem}\label{repr-thm}
Every element $f$ of $\mathcal{H}_2$ can be written in the form
$$f=\alpha^{a}\nu^{b}\beta^{c}\prod \left(\varphi\nu^{s_i}\beta^{r_i} \right)=\alpha^{a}\nu^{b}\beta^{c}\left(\varphi\nu^{s_n}\beta^{r_n} \right)\circ\cdots\circ\left(\varphi\nu^{s_1}\beta^{r_1} \right) $$ where $a, b, s_i=0,1$ and $c, r_i\in\mathbb{Z}$.
Further such a representation of $f$ is unique.
\end{theorem}
\begin{proof}
Let $f \in \mathcal{H}_2$ and $f(P) = Q$. If $Q=P$, the algorithm exits. Now since any element of $\mathcal{H}_2$ which fixes $P$ should have the form $\alpha^{a}\nu^{b}\beta^{c}$, where $a, b, c$ are as in the statement of the theorem, we are done. Note that in writing the prefix, the description of the stabilizer of $P$ in \cite{scharlemann2003automorphisms} and the commutativity relations of the generators are used. If $Q \neq P$, the above algorithm starts in step 2 with an application of an integral (possibly zero) power of $\beta$. Once $n_{AQ}<n_{BQ}+n_{CQ}$ the algorithm reaches step 3. At this stage either a $\varphi$ or a $\varphi\nu$ is applied either of which can be expressed as $\varphi\nu^{s}$, where $s=0,1$. After this application of $\varphi\nu^{s}$, by Theorem \ref{phi-thm}, the updated values of $n_{AQ}, n_{BQ}, n_{CQ}$ satisfy the inequality $n_{AQ}>n_{BQ}+n_{CQ}$. Then the algorithm either exits ($Q=P$) or continues with applications of powers of $\beta$. Since at each step $\mathcal{C}(Q)$ decreases, the algorithm has to terminate. Once the algorithm exits, $f$ can have a prefix of the form $\alpha^{a}\nu^{b}\beta^{c}$. So the algorithm expresses $f$ in the above form.

For the uniqueness, note that barring the prefix, for each factor $\varphi\nu^{s_i}\beta^{r_i}$, $\mathcal{C}(\varphi\nu^{s_i}\beta^{r_i}(Q))> \mathcal{C}(Q)$. So $\varphi\nu^{s_i}\beta^{r_i}(Q)$ cannot be equal to $Q$. So $f$ cannot have two different expressions of the above form.  
\end{proof}
\subsection{Illustration of the algorithm}
	Here we present a couple of examples of two reducing spheres and observe the application of the above algorithm. Consider the following examples:
	\begin{figure}[h!]
		\begin{minipage}{.49\linewidth}
			\begin{center}
			\includegraphics[width=.8\linewidth]{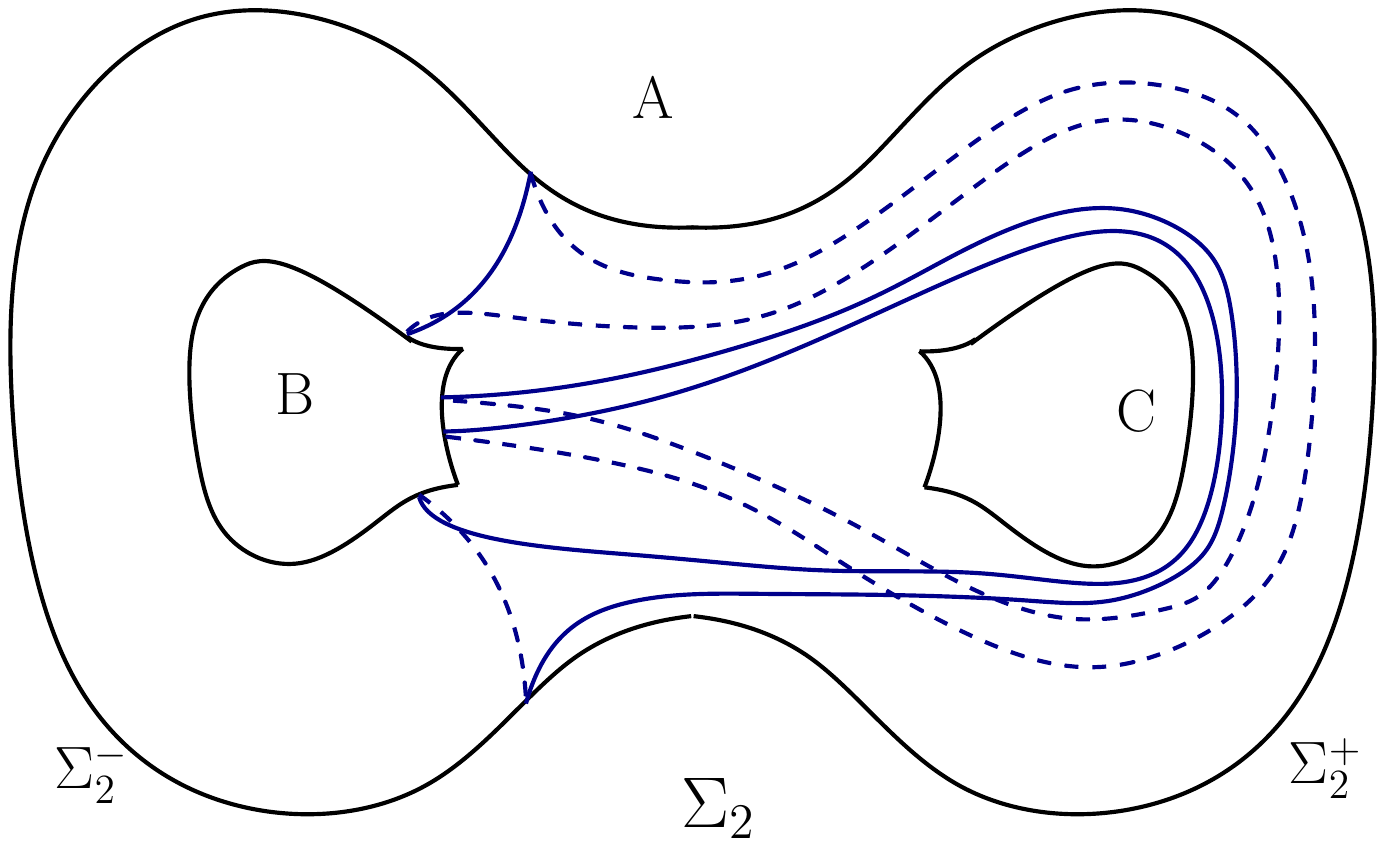}
			\caption{Example 1}\label{exm1}
			\end{center}
		\end{minipage}
	\begin{minipage}{.49\linewidth}
		\begin{center}
		\includegraphics[width=.8\linewidth]{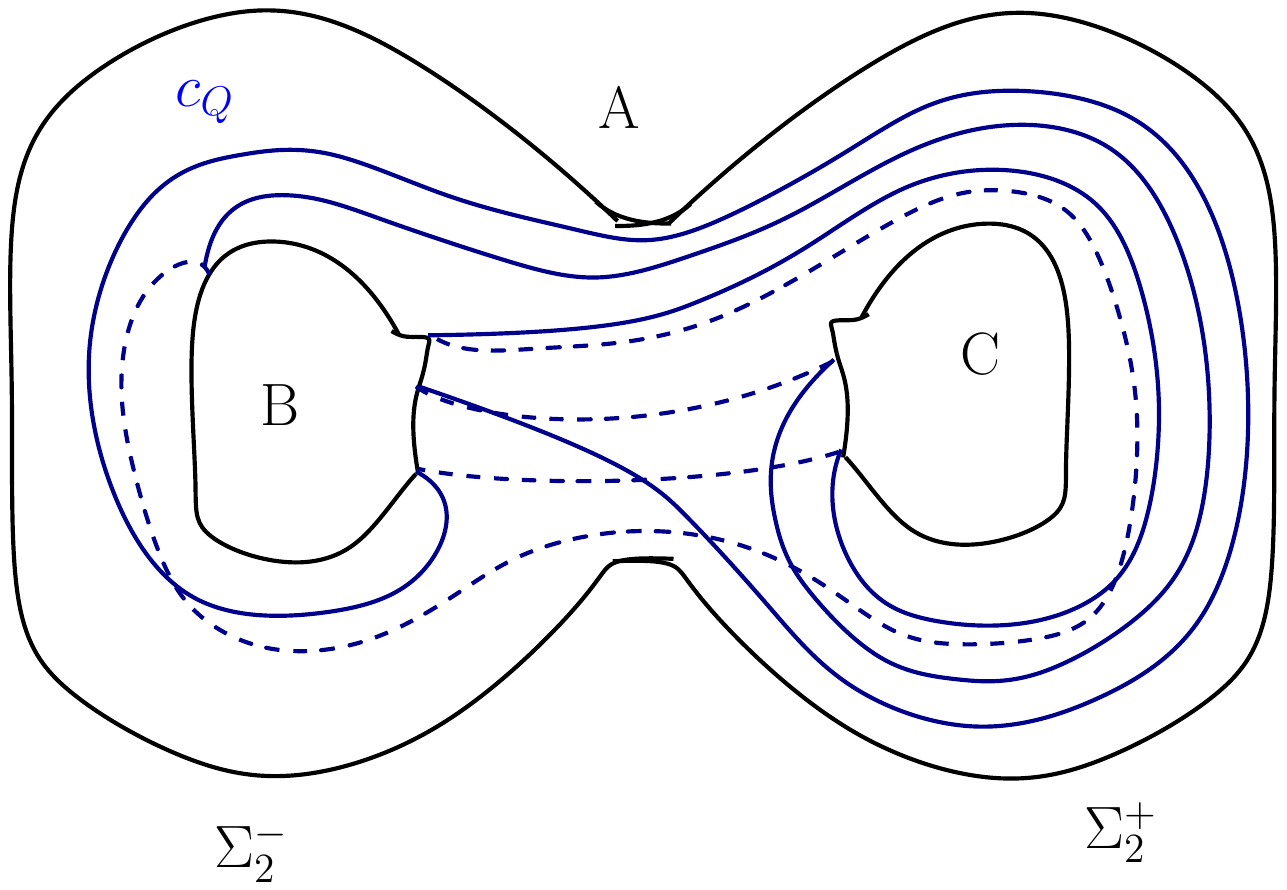}
		\caption{Example 2}\label{exm2}
		\end{center}		
	\end{minipage}
	\end{figure}
	
	In figure \ref{exm1}, $n_{AQ}<n_{BQ}+n_{CQ}$, and $n_{BQ}>n_{CQ}$. So we apply $\varphi$. In figure \ref{exm2} we have, $n_{BQ}+n_{CQ}=6>0=n_{AQ}$ also $n_{BQ}>n_{CQ}$. Here too we apply $\varphi$. On application of $\varphi$
	we get the spheres in figure \ref{exm1a} and \ref{exm2a} respectively.
	\begin{figure}[h!]
		\begin{minipage}{.49\linewidth}
			\centering\includegraphics[width=.8\linewidth]{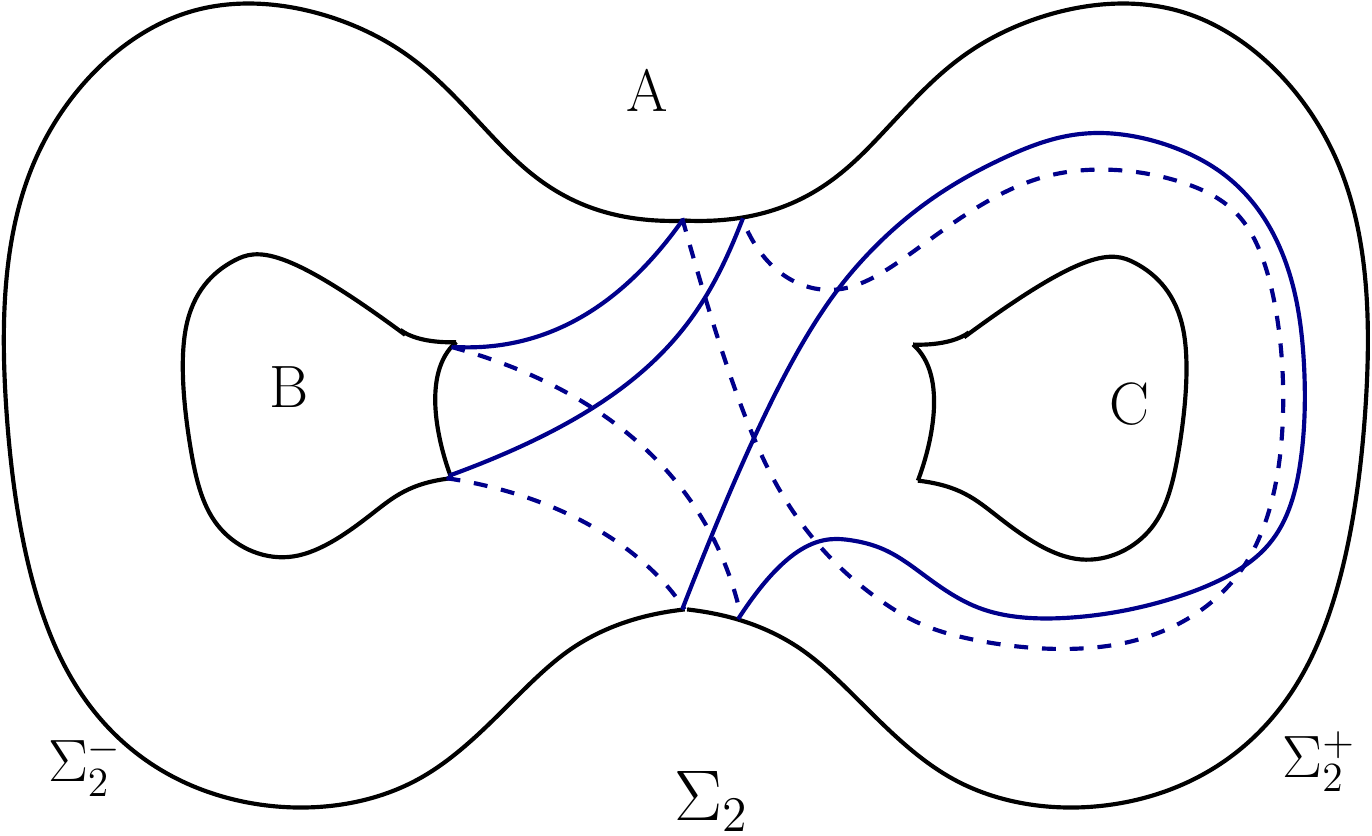}
			\caption{Example 1}\label{exm1a}
		\end{minipage}
		\begin{minipage}{.49\linewidth}
			\centering\includegraphics[width=.8\linewidth]{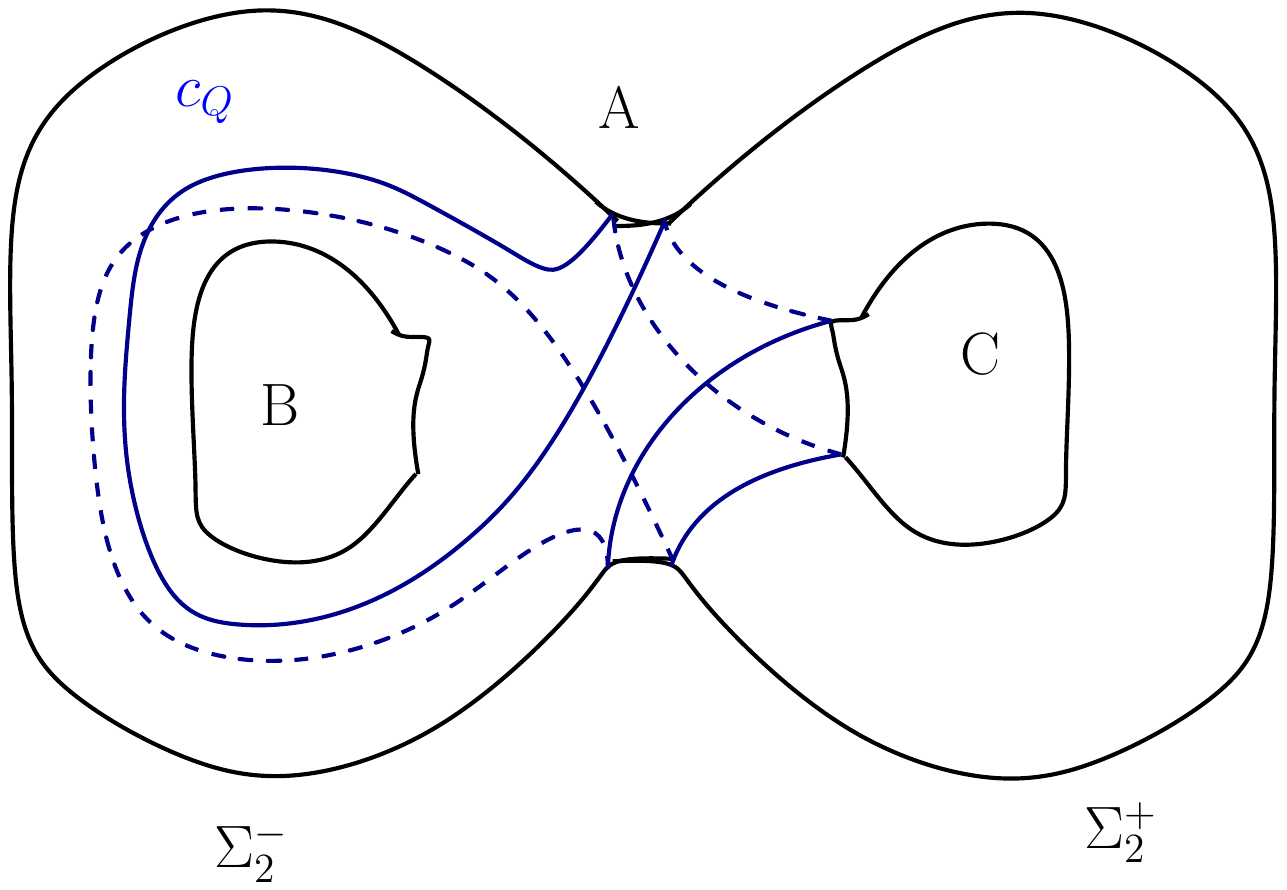}
			\caption{Example 2}\label{exm2a}
		\end{minipage}
	\end{figure}

	Let $\varphi(Q)=R$. Then from both figure \ref{exm1a} and figure \ref{exm2a}, we have $n_{AR}=4>2=n_{BR}+n_{CR}$. So now we apply $\beta$ or $\beta^{-1}$ suitably. For instance here we apply $\beta^{-1}$ in both cases. The result is presented in figure \ref{exm1b} and figure \ref{exm2b} respectively.
	\begin{figure}[h!]
		\begin{minipage}{.49\linewidth}
			\centering\includegraphics[width=.8\linewidth]{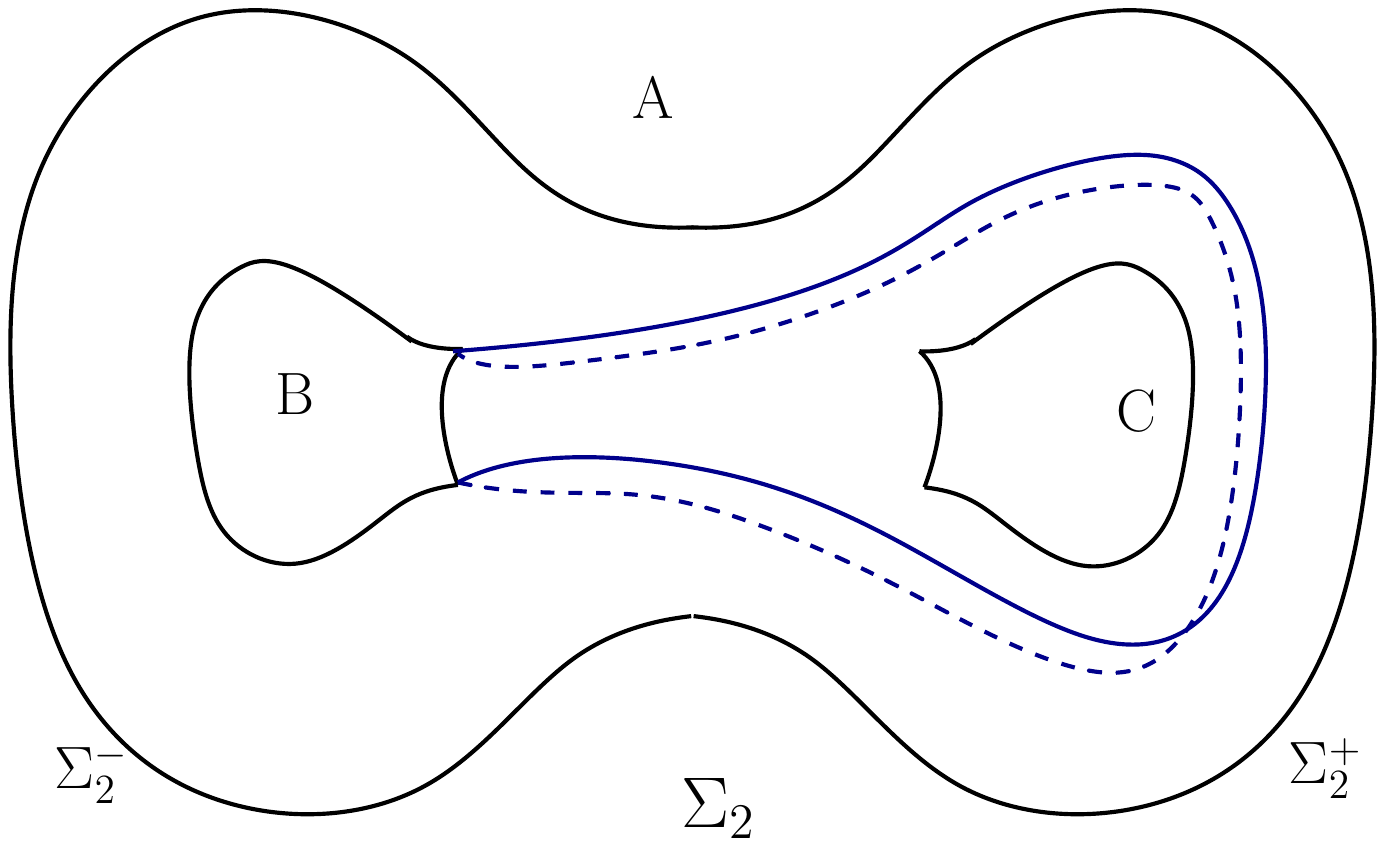}
			\caption{Example 1}\label{exm1b}
		\end{minipage}
		\begin{minipage}{.49\linewidth}
			\centering\includegraphics[width=.8\linewidth]{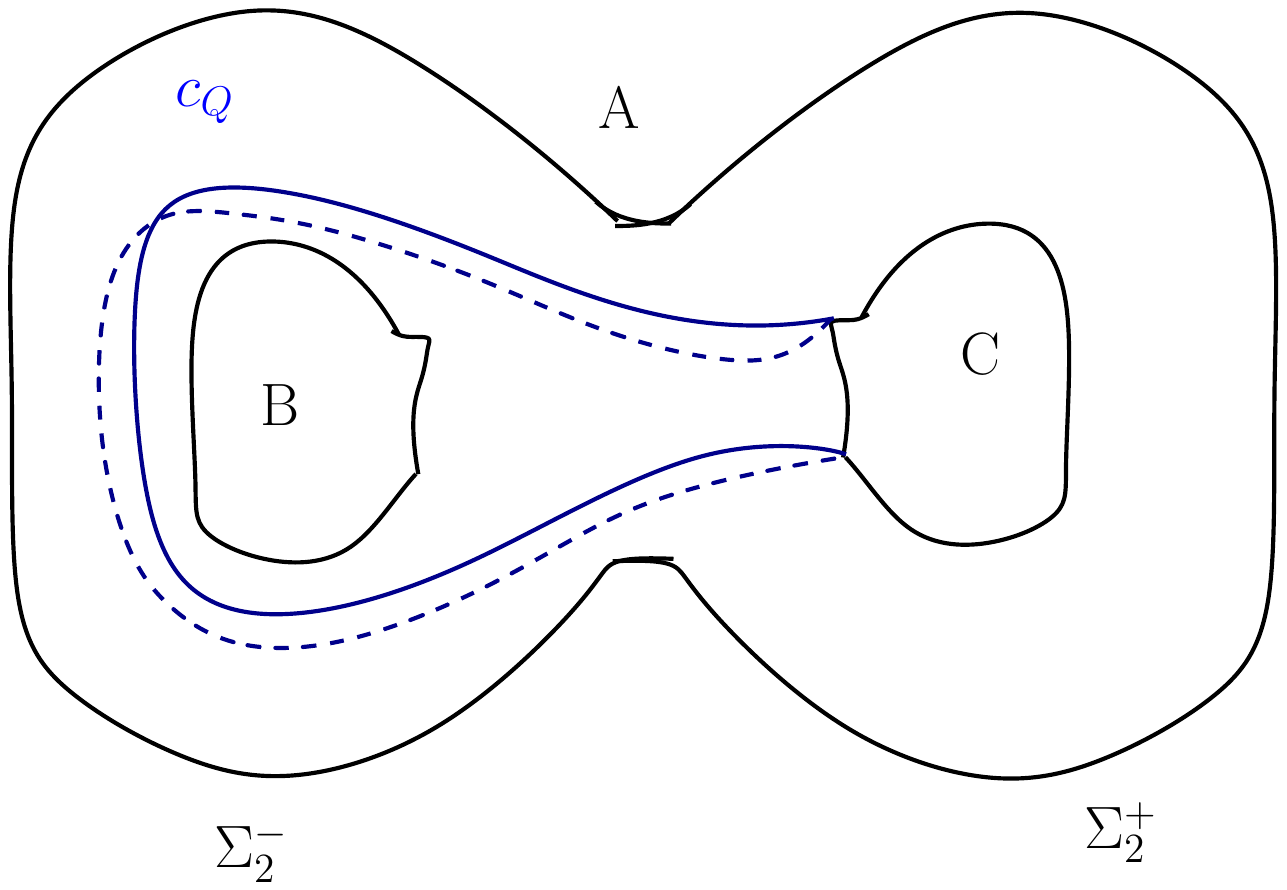}
			\caption{Example 2}\label{exm2b}
		\end{minipage}
	\end{figure}
	Now if $\beta^{-1}(R)=S$, in both cases we have $n_{AS}=0$. In the first case, we have $n_{BS}=2, n_{CS}=0$ whereas in the other we have $n_{BS}=0, n_{CS}=2$. So in both cases, $n_{AS}<n_{BS}+n_{CS}$. In first case $n_{BS}>n_{CS}$ and we apply $\varphi$ again whereas in the second one $n_{CS}>n_{BS}$ and so we apply $\varphi\nu$. It can be easily calculated that in both cases we are left with the standard curve $c_P$. 
	
	Therefore the automorphism that takes the first one to the standard is given by $f=\varphi\beta^{-1}\varphi$ and the one that takes the second one to standard is given by $g=\varphi\nu\beta^{-1}\varphi$.
	
		\subsection{The automorphism $\delta$ is in $\langle G_2 \rangle$}
	From the description of $\alpha$ and $\beta$ they are already identical with the corresponding generators in \cite{scharlemann2003automorphisms}. Also $\gamma=\nu\alpha$.
Consider the subset of $\mathcal{MCG}(\Sigma_2)$ given by $G_2 = \{\nu, \alpha, \beta, \varphi \}.$
	We will show that the automorphism $\delta$ described in  \cite{scharlemann2003automorphisms} is in $\langle G_2 \rangle$.
	
	\begin{proposition}
		The automorphism $\delta$ from \cite{scharlemann2003automorphisms} is generated by $\nu$ and $\varphi$ and we can express $\delta$ as $\delta = \nu^{-1}\varphi\nu\varphi = (\nu\varphi)^2$.
	\end{proposition} 
	
	\begin{proof}
		From the earlier discussion, we have $\varphi(B)=A$. We also have
		$$\varphi(A)=B, \varphi(C)=C, \varphi(X)=Y, \varphi(Y)=X, \text{ and } \varphi(Z)=Z.$$ Now since $\gamma$ exchanges the two genus one summands and leaves $X$ and $A$ invariant therefore
		$$\nu\varphi(A)=C, \nu\varphi(B)=A, \nu\varphi(C)=B,\nu\varphi(X)=Z, \nu\varphi(Y)=X  \text{ and } \nu\varphi(Z)=Y.$$
		Now if $\psi= \nu\varphi\nu\varphi$, then 
		$$\psi(A)=B, \psi(B)=C, \psi(C)=A,\psi(X)=Y,\psi(Y)=Z,   \text{ and } \psi(Z)=X.$$
		Now from the description of $\delta$ we have
		$$\delta(A)=B,\delta(B)=C, \delta(C)=A, \delta(X)=Y,\delta(Y)=Z \text{ and } \delta(Z)=X.$$
		Therefore, $\psi^{-1}\delta$ fixes all the above mentioned loops on $\Sigma_2$ and also fixes $\Sigma_2'$ and $\Sigma_2''$. But that implies $\psi^{-1}\delta=1$ i.e. $\delta=\psi=(\nu\varphi)^2$.\\
		This completes the proof. 
	\end{proof}
	
	Therefore, elements in $G_2$ generates the elements of $\mathcal{H}_2$ proposed by \cite{scharlemann2003automorphisms}. Thus this gives another proof of the fact that $G_2$ generates $\mathcal{H}_2$.
\bibliographystyle{plainnat}
\nocite{*}
\bibliography{bib}

\end{document}